\documentclass{amsart}[12pt]
\usepackage{amsmath, amsthm, amscd, amsfonts,}
\usepackage{graphicx,float}

\usepackage{amssymb}

\usepackage{pb-diagram}

\DeclareMathOperator{\dom}{dom}

\DeclareMathOperator{\rge}{rge}

\DeclareMathOperator{\lh}{lh}

\DeclareMathOperator{\crit}{crit}

\DeclareMathOperator{\ot}{ot}

\DeclareMathOperator{\cf}{cf}

\theoremstyle{plain}

\newtheorem{theorem}{Theorem}

\newtheorem{lemma}{Lemma}[section]

\newtheorem{corollary}[lemma]{Corollary}

\newtheorem{definition}[lemma]{Definition}

\newtheorem{claim}{Claim}[lemma]

\newtheorem{subclaim}{Subclaim}[claim]

\theoremstyle{remark}

\newtheorem{remark}[lemma]{Remark}

\begin{document}

\title[Collapsing the cardinals of $HOD$]{Collapsing the cardinals of $HOD$}

\author[ J. Cummings; Sy D. Friedman and M. Golshani]{ James Cummings; Sy David Friedman and Mohammad Golshani}

\thanks{Cummings was partially supported by
   NSF grant DMS-1101156, and thanks NSF for their support. He also thanks the University of Vienna
   for inviting him as a Guest Professor in January 2013.}

\thanks{Friedman would like to thank the FWF (Austrian Science Fund) for its support
through Project P23316-N13.}

\thanks{Golshani's research was in part supported by a grant from IPM (No. 91030417). He also thanks
 the European Science Foundation for their support through the grant 5317 within the INFTY program.}

\thanks{The authors thank Hugh Woodin for a very helpful discussion of this problem,
     including the suggestion to use supercompact Radin forcing. In an email exchange
     with the first author \cite{Woodin-private-communication} Woodin suggested
     an alternative to using our projected forcing $\mathbb{R}_w^{\rm proj}$, but we have opted to
     retain our original formulation.}

\maketitle

\begin{abstract}
  Assuming that GCH holds and $\kappa$ is $\kappa^{+3}$-supercompact, we construct a generic extension
  $W$ of $V$ in which $\kappa$ remains strongly inaccessible and $(\alpha^+)^{HOD} < \alpha^+$ for every infinite
  cardinal $\alpha < \kappa$. In particular the rank-initial segment $W_\kappa$ is a model of ZFC in which
  $(\alpha^+)^{HOD} < \alpha^+$ for  every infinite cardinal $\alpha$.
\end{abstract}

\section{Introduction}

 The study of {\em covering lemmas} has played a central role in set theory. The original Jensen covering lemma for $L$ \cite{jensen-covering}
 states that if $0^\sharp$ does not exist then every uncountable set of ordinals in $V$ is covered by a set of ordinals in
 $L$ with the same $V$-cardinality. The restriction to uncountable sets is necessary because of the example of Namba forcing,
 which preserves $\aleph_1^V$ but changes the cofinality of $\aleph_2^V$ to $\omega$.  In combination with the ``internal''
 theory of $L$ coming from fine structure \cite{jensen-finestrux}, the Jensen covering lemma is a powerful tool for proving lower bounds
 in consistency strength.

 If $V_0$ is an inner model of $V_1$, then {\em Jensen covering holds between $V_0$ and $V_1$} if
 every uncountable set of ordinals in $V_1$ is covered by a set of ordinals in $V_0$ with the same $V_1$-cardinality. The example of
 Prikry forcing shows that Jensen's covering lemma does not directly generalise to Kunen's model $L[\mu]$.
 If $V= L[\mu]$ and $G$ is Prikry generic over $V$ then cardinals are preserved in $V[G]$,
  the Prikry generic set cannot be covered by any set in $L[\mu]$ of cardinality less than $\kappa$, and
 $0^\dagger$ does not exist in $V[G]$.

 {\em Weak covering holds between $V_0$ and $V_1$} if $(\lambda^+)^{V_0} = (\lambda^+)^{V_1}$ for every singular cardinal
 $\lambda$ in $V_1$. It is not hard to see that Jensen covering implies weak covering, the key point is that
 if $(\lambda^+)^{V_0} < (\lambda^+)^{V_1}$ then $\cf^{V_1}((\lambda^+)^{V_0})< \lambda$. Note that for example
 weak covering holds between $L[\mu]$ and a Prikry extension of it. Jensen and Steel \cite{jensen-steel}
  have shown that if there is no inner model with a Woodin cardinal then weak covering holds over
 the core model for a Woodin cardinal.

 The Jensen-Steel covering result implies that we will require the strength of a Woodin cardinal to obtain any
 failure of weak covering between a model $V$ and a generic extension $V[G]$.  To see this suppose there is no inner model with
 a Woodin cardinal and $\lambda$ is singular in $V[G]$. Then
 by Jensen-Steel $(\lambda^+)^{V[G]} = (\lambda^+)^{K^{V[G]}}$ where $K$ is the core model for one Woodin cardinal.
 The core model $K$ is generically absolute so $K^{V[G]} = K^V \subseteq V$, and hence
 $(\lambda^+)^{V[G]} = (\lambda^+)^{K^{V[G]}} = (\lambda^+)^{K^V} = (\lambda^+)^V$.

  It is possible to force a failure of weak covering between $V$ and some generic extension from a Woodin cardinal,
  using the stationary tower
  forcing \cite{larson}.  Let $\delta$ be Woodin, let $S = \{ \alpha < \aleph_{\omega+1} : \cf(\alpha) = \omega \}$,
  and let $G$ be generic for the stationary tower forcing ${\mathbb P}_{<\delta}$ with $S \in G$. Then $\bigcup S = \aleph_{\omega+1}^V$,
  and if $j : V \rightarrow M \subseteq V[G]$ is the generic embedding associated with $G$ then by a standard fact
  $j[\aleph_{\omega+1}] \in j(S)$. Since $j(S)$ is a set of ordinals this implies that
  $\crit(j) \ge  \aleph_{\omega+1}^V$. Since $V[G] \models {}^{<\delta} M \subseteq M$, the cardinals
  of $V$ agree with those of $V[G]$ up to $\aleph_\omega$.
  Moreover $j[\aleph_{\omega+1}^V] = \aleph_{\omega+1}^V$ and hence $\cf^{M}(\aleph_{\omega+1}^V) = \omega$,
  so that in particular $\aleph_{\omega+1}^V < (\aleph_{\omega+1})^{V[G]}$.  So $\aleph_\omega^V$ is a singular cardinal
  of $V[G]$ whose successor is not computed correctly by $V$, a failure of weak covering.

  This argument is not the last word, because a key point in the theory of covering lemmas is that weak covering should hold over
  a ``reasonably definable'' inner model. Of course the term  ``reasonably definable'' is quite vague, so let us
  stipulate that a reasonably definable model should at least be ordinal definable and have an ordinal definable
  well-ordering. If $M$ is such a model then $M \subseteq HOD$, because every element of $M$ can be defined
  in $V$ as ``the $\alpha^{\rm th}$ element of $M$'' for some ordinal $\alpha$; what is more, HOD itself is such a model.
  It follows that a failure of weak covering over $HOD$ will also be a failure of weak covering over
  every reasonably definable inner model. The main result of this paper produces a model where
  $(\alpha^+)^{HOD} < \alpha^+$ for every infinite cardinal $\alpha$, which is in a certain sense
  the ultimate failure of weak covering.

   The problem of arranging that $(\alpha^+)^{HOD} < \alpha^+$ has a very easy solution for a single regular cardinal $\alpha$.
   If we force with the Levy collapse $Coll(\alpha, \alpha^+)$ then we obtain a generic extension
   $V[G]$ such that $\alpha$ is still regular,  $(\alpha^+)^V < (\alpha^+)^{V[G]}$ and  (by the homogeneity of the collapse)
   $HOD^{V[G]} \subseteq V$. Work of Dobrinen and Friedman \cite{dobrinen-friedman} shows that it is  consistent
   that $(\alpha^+)^{HOD} < \alpha^+$ for every regular $\alpha$, and that this can hold in the presence of very large
   cardinals (for example superstrong cardinals). This is in contrast to a covering result by Steel \cite{steel}
   stating that if there is no inner  model with a Woodin cardinal and $\kappa$ is measurable with some
   normal measure $\mu$, then $\{ \alpha:  (\alpha^+)^K = \alpha^+ \} \in \mu$.

   The stationary tower argument given above does not provide a failure of weak covering over HOD,
    because the stationary tower forcing is very inhomogeneous and there is no
   reason to believe that $HOD^{V[G]} \subseteq V$. In fact the only arguments known to the authors
   for obtaining a failure of weak covering over $HOD$ involve some use of supercompactness:
   we will outline one such argument below in Section \ref{preliminaries}.

   To prove our main result we will start with a suitable large cardinal $\kappa$, which will still be inaccessible
   in our final model.  We will  first build generic extensions $V_0^*$ and
   $V^*$ with $V \subseteq V_0^* \subseteq V^*$, such that $(\alpha^+)^V = (\alpha^+)^{V_0^*} < (\alpha^+)^{V^*}$
   for almost all (in the sense of the club filter) infinite $V^*$-cardinals $\alpha < \kappa$.
   We will also arrange that $HOD^{V^*} \subseteq V_0^*$,   so that in $V^*$ we have $(\alpha^+)^{HOD} < \alpha^+$
   for almost all infinite cardinals $\alpha < \kappa$.
   The model $V^*$ will be obtained by supercompact Radin forcing, and $V_0^*$ will be a submodel generated
   by a projected version of supercompact Radin forcing; then $V^*$ is a generic extension of $V_0^*$ by
   a quotient forcing, and the  projected forcing is constructed so that we have enough homogeneity for the quotient forcing
   to argue that $HOD^{V^*} \subseteq V_0^*$. To finish we will force over $V^*$ with a sufficiently homogeneous
   iteration of Levy collapsing posets,
   and produce a model $W$ such that $HOD^W \subseteq HOD^{V^*}$ and $(\alpha^+)^V = (\alpha^+)^{V_0^*} < (\alpha^+)^W$ for every
   $W$-cardinal $\alpha < \kappa$. Since $\kappa$ is inaccessible in $W$, the initial segment $W_\kappa$ will
   be a model of ZFC in which $(\alpha^+)^{HOD} < \alpha^+$ for  every infinite cardinal $\alpha$.

   We do not know the exact consistency strength of the assertion  ``$(\alpha^+)^{HOD} < \alpha^+$ for  every infinite cardinal $\alpha$''.
   Since it implies that weak covering fails over every reasonably definable inner model at every singular cardinal,
   it is certainly very strong, and presumably it implies the consistency of many Woodin cardinals and of strong forms
   of determinacy.  Our work in this paper gives the upper bound ``GCH and $\kappa$ is $\kappa^{+3}$-supercompact'',
   which seems quite reasonable since in the current state of knowledge we need  ``GCH and $\kappa$ is $\kappa^+$-supercompact'' to
   get the consistency of ``$(\alpha^+)^{HOD} < \alpha^+$ for some singular cardinal $\alpha$''.

   Our notation is quite standard. We will write the pointwise image of a function $f$ on
   a set $A \subseteq \dom(f)$ as $f[A]$ rather than $f`` A$.  If $\mathbb P$ is a forcing poset
   with $p \in {\mathbb P}$ then we write ${\mathbb P} \downarrow p$ for $\{ q \in {\mathbb P}: q \le p \}$.
   If $s$ and $t$ are sequences then $s^\frown t$ is the concatenation of $s$ and $t$.

\section{Preliminaries on supercompact Prikry forcing and Radin forcing} \label{preliminaries}

   To motivate our use of supercompact Radin forcing, in this section we give a discussion of supercompact Prikry forcing
   and the original (non-supercompact) form of  Radin forcing.
   Supercompact Prikry forcing will give us a way of arranging that $(\kappa^+)^{HOD} < \kappa^+$
   for a single singular cardinal $\kappa$.

   Recall that Prikry forcing is defined from a normal measure $U_0$ on a measurable cardinal $\kappa$. It  is a
   $\kappa^+$-cc~forcing poset which adds an $\omega$-sequence cofinal in $\kappa$ without adding bounded subsets
   of $\kappa$. Supercompact Prikry forcing and Radin forcing represent generalisations of Prikry forcing
   in different directions: supercompact Prikry forcing adds an increasing $\omega$-sequence of elements of
   $P_\kappa \lambda$ whose union is $\lambda$ for some $\lambda \ge \kappa$, while Radin forcing adds closed unbounded subsets of $\kappa$
   of order types greater than or equal to $\omega$.

   The supercompact Radin forcing which we define in Section \ref{supercompactradin} is a common generalisation of
    supercompact   Prikry forcing and Radin forcing. Proofs of the various assertions we make about
    supercompact   Prikry forcing and Radin forcing can mostly be found (in the more general setting of supercompact Radin forcing)
    in that section.

\subsection{Supercompact Prikry forcing} \label{supercompactprikry}

   Supercompact Prikry forcing was introduced by Magidor \cite{magidor} in his work on the Singular Cardinals Hypothesis.

   Let $\kappa$ be $\lambda$-supercompact for some regular cardinal $\lambda \ge \kappa$, and let $U$ be a supercompactness measure
   on $P_\kappa \lambda$. It will be convenient to identify a subset on which $U$ concentrates:

\begin{definition}  Let $A(\kappa, \lambda)$  be the set of
 those $x \in P_\kappa \lambda$ such that $x \cap \kappa$ is an inaccessible cardinal and $\ot(x)$ is a regular
   cardinal with $\ot(x) \ge  x \cap \kappa$.
\end{definition}

   We claim that $U$ concentrates on $A(\kappa, \lambda)$. The argument will be the prototype for several later arguments,
   and will use the following standard fact:
   if $j: V \rightarrow M = Ult(V, U)$ is the ultrapower map associated with $U$ and $B \subseteq P_\kappa \lambda$, then
   $B \in U$ if and only $j[\lambda] \in j(B)$. Since $j[\lambda] \cap j(\kappa) = \kappa$ and
   $\ot(j[\lambda]) = \lambda$, it is clear that $j[\lambda] \in j( A(\kappa, \lambda) )$.

   Let $A = A(\kappa, \lambda)$ and define $\kappa_x = x \cap \kappa$, $\lambda_x = \ot(x)$ for $x \in A$; the functions
   $x \mapsto \kappa_x$ and $x \mapsto \lambda_x$ represent $\kappa$ and $\lambda$ respectively in the ultrapower
   of $V$ by $U$, and it is often useful to view $\kappa_x$ and $\lambda_x$ as ``local'' versions of $\kappa$ and $\lambda$.
   We equip $A$ with a strict partial ordering $\prec$ by defining $x \prec y$ if and only if $x \subseteq y$ and $\lambda_x < \kappa_y$.

\begin{remark} If $x \in A(\kappa, \lambda)$ then $\kappa_x = \min(ON \setminus x)$ and $\lambda_x = \ot(x)$,
  so that both $\kappa_x$ and $\lambda_x$ can be computed without knowing the values of $\kappa$ and $\lambda$.
\end{remark}

\begin{definition}
   The {\em supercompact Prikry forcing} ${\mathbb P}_U$ defined from $U$ has conditions $(s, B)$ where $B \in U$ with $B \subseteq A$,
   and $s$ is a finite $\prec$-increasing sequence drawn from $A$. The sequence $s$ is the {\em stem} of the condition,
   and the measure one set $B$ is the {\em upper part}.  The ordering is just like Prikry forcing,
   that is to say the condition $(s, B)$ is extended by prolonging $s$ using elements from $B$ and shrinking $B$:
   a {\em direct extension} of $(s, B)$ is an extension of the form $(s, C)$ for some $C \subseteq B$.
\end{definition}

\begin{remark}
   By a classical theorem of Solovay we have $\lambda^{<\kappa} = \lambda$, so that there are only
   $\lambda$ possible stems.
 If $\lambda = \kappa$ then $A(\kappa, \lambda) = \kappa$, and
   ${\mathbb P}_U$ is  just the standard Prikry forcing defined from a normal measure on $\kappa$.
\end{remark}

   The generic object for the forcing poset ${\mathbb P}_U$  is a $\prec$-increasing $\omega$-sequence
   $\langle x_i :i < \omega \rangle$ such that  $\bigcup_i x_i = \lambda$. The poset satisfies the
  {\em Prikry lemma}, that is to say that any question about the generic extension can be decided
   by a direct extension. This implies that ${\mathbb P}_U$ adds no bounded subsets of $\kappa$, in particular
   $\kappa$ remains a cardinal. By contrast with the standard Prikry forcing
   at $\kappa$, which is $\kappa^+$-cc and preserves all cardinals, forcing with ${\mathbb P}_U$ will collapse $\lambda$
   to have cardinality $\kappa$. Since any two conditions with the same stem are compatible,
   and there are only $\lambda$ stems, we see that ${\mathbb P}_U$ is $\lambda^+$-c.c.~and so it is only
   cardinals $\mu$ with $\kappa < \mu \le \lambda$ that are collapsed.

   The main point in the proof of the Prikry lemma (see Lemma \ref{sccprikrylemma})  is that $U$ enjoys a form of normality.
   If $s$ is a stem and $y \in A$ then we write ``$s \prec y$'' for the assertion that either $s$ is empty or the
   last entry in $s$ is below $y$ in the $\prec$ ordering: normality states that  if $I$ is a set of stems and
    $\langle B_s : s \in I \rangle$  is such  that $B_s \in U$ for all $x$, and we define
   the diagonal intersection $\Delta_s B_s = \{ y \in A : \forall s \prec y \; y \in B_s \}$, then
   $\Delta_s B_s \in U$.

   We can use supercompact Prikry forcing to obtain a model where $(\alpha^+)^{HOD} < \alpha^+$ for a single
   singular cardinal $\alpha$. Let us assume for simplicity that $\lambda = \kappa^+$.
   Let  $\langle x_i :i < \omega \rangle$ be a ${\mathbb P}_U$-generic sequence and let $U_0$ be the projection of
   $U$ to a measure on $\kappa$ via the map $x \mapsto \kappa_x$. Using the well-known criterion for Prikry genericity,
   one can show that $\langle \kappa_{x_i} : i < \omega \rangle$  is generic for ${\mathbb P}_{U_0}$,
    the Prikry forcing at $\kappa$ defined from the measure $U_0$.

  So starting with $G$ which is ${\mathbb P}_U$-generic, we  obtain a chain of models $V \subseteq V[G_0] \subseteq V[G]$
   where $G_0$ is ${\mathbb P}_{U_0}$-generic, and $(\kappa^+)^V = (\kappa^+)^{V[G_0]} < (\kappa^+)^{V[G]}$. Once we have shown that
   $HOD^{V[G]} \subseteq V[G_0]$, it will follow that in the model $V[G]$ weak covering fails for HOD at $\kappa$.
  This can be proved (see Lemma \ref{permutationfact}) using permutations of $\kappa^+$ that fix all points below $\kappa$;
  the key idea is (roughly speaking) that such permutations induce many automorphisms of ${\mathbb P}_U$ which
   commute with the operation of projecting a ${\mathbb P}_U$-generic filter down to a ${\mathbb P}_{U_0}$-generic filter.

\subsection{Radin forcing} \label{vanillaradin}

  Radin forcing \cite{radin} is a generalisation of Prikry forcing in which closed unbounded sets of order types greater than or equal
  to $\omega$ are added to a large cardinal $\kappa$. We will outline the basic theory of Radin forcing.

\begin{definition}  A {\em $\kappa$-measure sequence} is a sequence $w$ such that
   $w(0) = \kappa$,  and $w(\alpha)$ is a $\kappa$-complete measure on $V_\kappa$ for
   $0 < \alpha < \lh(w)$.
\end{definition}

\begin{definition}
  Given an elementary embedding $j: V \rightarrow M$ with ${\rm crit}(j) = \kappa$, we derive
  a  $\kappa$-measure sequence $u_j$ by setting $u_j(0) = \kappa$ and then
   $u_j(\alpha) = \{ X \subseteq V_\kappa: u_j \restriction \alpha \in j(X) \}$
  for $\alpha > 0$, continuing for as long as $\alpha < j(\kappa)$ and  $u_j \restriction \alpha \in M$.
\end{definition}

  Each $u_j(\alpha)$ for $\alpha > 0$ is
  a $\kappa$-complete non-principal ultrafilter on $V_\kappa$, concentrating on objects which
  ``resemble'' $u_j \restriction \alpha$. Note that $u_j(1)$ is essentially the usual normal measure
  on $\kappa$ derived from $j$; it concentrates on sequences of length one rather than ordinals because
  it is generated by the sequence $u_j \restriction 1 = \langle \kappa \rangle$ rather than the
  ordinal $\kappa$.
  In this informal discussion, we will sometimes ignore the distinction between the ordinal $\alpha$ and
  the sequence $\langle \alpha \rangle$.

\begin{definition}
 If $w$ is an initial segment of $u_j$ then
  we say that $j$ is a {\em constructing embedding} for $w$.
\end{definition}

\begin{definition}
We define a class $\mathcal{U}_{\infty}$ of {\em good measure sequences} as follows:
\begin{itemize}
\item $\mathcal{U}_{0}$  is the class of $w$ such that for some $\kappa$,
   $w$ is a $\kappa$-measure
   sequence which has a constructing embedding.
\item  $\mathcal{U}_{n+1}=\{ w \in \mathcal{U}_{n}: \mbox{for all nonzero $\alpha < \lh(w)$, $w(\alpha)$ concentrates on $\mathcal{U}_{n}$} \}$.
\item $\mathcal{U}_{\infty}=\bigcap_{n<\omega}\mathcal{U}_{n}$.
\end{itemize}
\end{definition}

  Note that by countable completeness, if $w$ is a good measure sequence then every measure in $w$
  concentrates on good measure sequences.
  By choosing $j$ to witness some modest degree of strength for $\kappa$, we may arrange that
  long initial segments of $u_j$ are good measure sequences. The key point here is that if $j: V \rightarrow M$ then
  there are extenders in $M$ which approximate $j$ sufficiently to serve as constructing embeddings (in $M$) for
  initial segments of $j$.

  Given a measure sequence $u$ with $\lh(u) > 1$  we define the Radin forcing ${\mathbb R}_u$ as follows.
\begin{definition}
 Conditions in ${\mathbb R}_u$  are
  sequences $p = \langle (u^i, A^i) : i \le n \rangle$ where each $u^i$ is a good measure sequence in $V_{u^{i+1}(0)}$, $u^n = u$,
  $u^i(0)$ increases with $i$, and either $\lh(u^i) = 1$ and $A^i = \emptyset$ or $\lh(u^i) > 1$ and
  $A^i \subseteq V_{u^i(0)}$ with  $A^i$ a set of good measure sequences which is measure one for every measure in $u^i$. A condition is extended
  by shrinking some of the sets $A^i$, and  interpolating new pairs $(w, B)$ such that $w \in A^i$ and $B \subseteq A^i \cap V_{w(0)}$
  between $(u^{i-1}, A^{i-1})$ and $(u^i, A^i)$ for some $i$.
\end{definition}

 The generic object can be viewed as a sequence of good measure sequences
  $\langle v_\alpha : \alpha < \mu \rangle$ where $\langle v_\alpha(0) : \alpha < \mu \rangle$ increases continuously with $\alpha$
 and is cofinal in $u(0)$: the condition $p$ above carries the information that each $v_\alpha$
  either appears among the $u^i$ or is the first entry in a pair which can legally be added to $p$.
 The uniform definition of the extension relation readily implies that if $\lh(v_\alpha) > 1$ then the
 sequence $\langle v_\beta : \beta < \alpha \rangle$ is  generic over $V$ for ${\mathbb R}_{v_\alpha}$.

  It is easy to see that if $p$ is a condition as above and $i < n$ with $\lh(u^i) > 1$,
  then ${\mathbb R}_u \downarrow p$ is isomorphic to
  a product of the form $R_{u^i} \downarrow q \times {\mathbb R}_u \downarrow r$ for suitable conditions
  $q \in {\mathbb R}_{u^i}$ and $r \in {\mathbb R}_u$. The forcing ${\mathbb R}_u$ is also $u(0)^+$-cc and satisfies a
  version of the Prikry lemma, stating as usual that any question about the generic extension can be decided just
  by shrinking measure one sets. It follows from these facts that forcing with ${\mathbb R}_u$ preserves all cardinals.

  As motivation, we outline a version of Radin forcing
  intended to add a cofinal and continuous sequence of type $\omega^2$ in a large cardinal $\kappa$.
  Consider the Radin forcing defined from a good measure sequence $u$ of length three, so
  $u$  has two measures $u(1)$ and $u(2)$. If we force below the  condition
  $(u, A)$ where $A$ consists of good measure sequences in $V_{u(0)}$ of length one or two,
  then we will obtain a generic sequence $\langle v_\alpha: \alpha < \omega^2 \rangle$.
  where $v_\alpha$ has length one (so is morally just an ordinal) for successor $\alpha$
  but has length two for limit $\alpha$.

  It can be shown that:
\begin{itemize}
\item For each $m < \omega$, the $\omega$-sequence $\langle v_{\omega \cdot m + n}(0) : n < \omega \rangle$ is Prikry generic for the
  Prikry forcing defined from $v_{\omega \cdot (m+1)}(1)$.
\item The $\omega$-sequence $\langle v_{\omega \cdot m} : m < \omega \rangle$ is generic for the version of Prikry forcing
  defined from the measure $u(2)$  (stems are finite with each entry a measure sequence of length two and critical points increasing,
   upper parts are $u(2)$-large sets).
\item  For any sequence $\langle \beta_i : i < \omega \rangle$ which consists of successor ordinals
   and is cofinal in $\omega^2$, the $\omega$-sequence $\langle v_{\beta_i}(0) : i < \omega \rangle$ is
   Prikry generic for the Prikry forcing defined from $u(1)$.
\end{itemize}

\section{Supercompact Radin forcing} \label{supercompactradin}

  Supercompact Radin forcing was introduced by Foreman and Woodin \cite{foreman-woodin} in their consistency proof
  for ``GCH fails everywhere''. In particular they proved that supercompact Radin forcing satisfies a version of the
  Prikry lemma, and can preserve large cardinals.  The main forcing of \cite{foreman-woodin}
   is rather complicated as it aims to interleave generic objects
  for Cohen posets along the generic sequence, and  conditions must contain machinery for constraining these generic objects.
  Other accounts of supercompact Radin forcing appear in the literature, for example Krueger \cite{krueger} has described a version
  constructed from a coherent sequence of supercompactness measures.

  In this section we define a version of supercompact Radin forcing. To make the paper self-contained,
  we will  prove all the properties of this forcing  which we will use. To motivate some technical aspects
   of the general definition, we will first define
  a special case of the forcing which adds a continuous and $\prec$-increasing $\omega^2$-sequence
  in $P_\kappa \lambda$.

\begin{definition} \label{sccradinuj}
   Let $\kappa \le \lambda \le \mu$, where $\lambda$ and $\mu$ are regular and $\kappa$ is $\mu$-supercompact.
 Let $j: V \rightarrow M$ witness $\mu$-supercompactness of $\kappa$,  and
  then define a sequence $u_j$ by the recursion $u_j(0) = j[\lambda]$,
  $u_j(\alpha) = \{ X : u_j \restriction \alpha \in j(X)  \}$ for as long as $\alpha < j(\kappa)$ and
   $u_j \restriction \alpha \in M$.
\end{definition}

\begin{remark}  If $\lambda = \kappa$ then we are just defining the kind of measure sequences
  constructed in Section \ref{vanillaradin}.
\end{remark}

  Recall from Section \ref{supercompactprikry} that
  we defined $A(\kappa, \lambda)$ as the set of $x \in P_\kappa \lambda$ such that
  $x \cap \kappa$ is an inaccessible cardinal and $\ot(x)$ is a regular cardinal
  with $\ot(x) \ge x \cap \kappa$.  We also defined $\kappa_x = x \cap \kappa$ and
  $\lambda_x = \ot(x)$ for such $x$.

\begin{definition} \label{skappalambda}
   Let $S(\kappa, \lambda)$ be the set of non-empty sequences $w$
   such that  $\lh(w) < \kappa$, $w(0) \in A(\kappa, \lambda)$ and $w(\alpha) \in V_\kappa$ for all $\alpha$ with $0 < \alpha< \lh(w)$.
\end{definition}

   It is easy to see that for every $\alpha < j(\kappa)$ such that the measure $u_j(\alpha)$ is defined,
   it concentrates on $S(\kappa, \lambda)$. When $\lh(w) = 1$, we will sometimes be careless
   about the distinction between the sequence $w = \langle w(0) \rangle$ and the set $w(0)$.

\begin{remark} We can refine the definition of the set $S(\kappa, \lambda)$ to find a smaller subset on which
   the measures on $u_j$ will concentrate, reflecting more of the properties of
   initial segments of $u_j$.  We do this in generality in Definition \ref{sccgoodmeasuresequence} below,
   for the purposes of the following example we just note that   $u_j(1)$ is (essentially) the $\lambda$-supercompactness
   measure on $P_\kappa \lambda$ derived from $j$ in the standard way, and so $u_j(2)$ concentrates on pairs
    $(x, w)$ where $x \in P_\kappa \lambda$ and $w$ is (essentially) a  supercompactness measure on $P_{\kappa_x} \lambda_x$.
\end{remark}

  Suppose that $u_j(1)$ and $u_j(2)$ are defined and let $u = u_j \restriction 3 = (j[\lambda], u_j(1), u_j(2))$.
  Conditions in the supercompact Radin  forcing to add an $\omega^2$-sequence in $P_\kappa \lambda$
  will be finite sequences $\langle (u^i, A^i) : 0 \le i \le n \rangle$
  where
\begin{itemize}
\item $u^i \in S(\kappa, \lambda)$ for $i < n$, and $\langle x^i : i < n \rangle$ is $\prec$-increasing
   where $x^i = u^i(0)$.
\item For $i < n$, either  $\lh(u^i) = 1$ and $A^i = \emptyset$, or $\lh(u^i) = 2$ and $u^i(1)$
   is a supercompactness measure on $A(\kappa_{x^i}, \lambda_{x^i})$ with $A^i \in u^i(1)$.
\item $u^n = u$, $A^n \subseteq S(\kappa, \lambda)$, $A^n \in u^n(1) \cap u^n(2)$ and
 it consists of sequences of length at most $2$.
\end{itemize}

   The ordering is basically as in the Radin forcing described in Section \ref{vanillaradin}, with one complication.
   If $i < n$ and $\lh(u^i) = 2$, then the elements of $A^i$ are in $A(\kappa_{x^i}, \lambda_{x^i})$
   but the objects we would like to interpolate between $(u^{i-1}, A^{i-1})$ and $(u^i, A^i)$ are
   elements $y \in A(\kappa, \lambda)$ such that $x^{i-1} \prec y \prec x^i$. The solution will be
   to use the order isomorphism $\pi: x^i \simeq \lambda_{x^i}$, so that $y$ can be legally be interpolated
   if $x^{i-1} \prec y \prec x^i$ and $\pi[y] \in A^i$.

  This gives some insight into the way that supercompact Radin forcing will work in general. An entry $w$
  on the generic sequence will have $x = w(0) \in A(\kappa, \lambda)$, and the remainder of $w$ will consist of
  measures on $S(\kappa_x, \lambda_x)$. The measures appearing on $w$ define a ``local'' supercompact Radin forcing
  for $P_{\kappa_x} \lambda_x$, and the role of $x$ will be to integrate the generic sequence for this ``local''
  forcing into the ``global'' sequence. On a related point,  in the definition of the supercompact
  Radin forcing from a  sequence $u$ the value of  $u(0)$ is actually irrelevant.

\subsection{Supercompact measure sequences}

  Before defining supercompact Radin forcing, we need to define the ``good measure sequences'' which will form the
  building blocks for this forcing. The reader should note that the terms
  {\em constructing embedding}, {\em (good) measure sequence},  will be used
  in a more general sense than in Section \ref{vanillaradin}.

\begin{definition}  \label{kappalambdasequence}
 A sequence  $u$ is a {\em $(\kappa,\lambda)$-measure sequence} if $u(0)$ is a set of ordinals
 with $\kappa = \min(ON \setminus u(0))$ and $\lambda = \ot( u(0) )$,   and
 $u(\alpha)$ is a $\kappa$-complete measure on $S(\kappa, \lambda)$ for all $\alpha$ with $0 < \alpha < \lh(u)$.
\end{definition}

  The sequence $u_j$ constructed in definition \ref{sccradinuj}  is an example of such a sequence.

\begin{definition} If $u$ is a $(\kappa, \lambda)$-measure sequence and $A \subseteq S(\kappa, \lambda)$, then
   $A$ is {\em $u$-large} if and only if $A \in u(\beta)$ for all $\beta$ with $0 < \beta < \lh(u)$.
\end{definition}

\begin{definition} \label{sccconstructingembedding}
  Given a $(\kappa, \lambda)$-measure sequence $u$,
 {\em $j$ is a constructing embedding for $u$} if and only if  $j$ witnesses that $\kappa$ is
  $\lambda$-supercompact, and
  for all $\alpha$ with $0 < \alpha < \lh(u)$ we have that $u_j(\alpha)$ is
  defined with $u_j(\alpha) = u(\alpha)$.
\end{definition}

 Note that possibly  $u(0) \neq u_j(0)$ in the last definition, which may seem surprising.
 We will discuss this point further after Lemma \ref{gettingulemma}. The basic issue is that
 if $j : V \rightarrow M$ witness $\mu$-supercompactness for some large $\mu > \lambda$,
 and we build a $(\kappa, \lambda)$-measure sequence $u_j$ then we should like
 initial segments of $u_j$ to have a constructing embedding {\em in $M$}; this will be true
 with our definition.

\begin{definition} \label{sccgoodmeasuresequence}

We define a class $\mathcal{U}^{\rm sup}_{\infty}$ of {\em good measure sequences} as follows:
\begin{itemize}

\item $\mathcal{U}^{\rm sup}_{0}$  is the class of sequences $u$
  such that for some regular cardinals $\kappa$ and $\lambda$ with $\kappa \le \lambda$,
   $u$ is a $(\kappa, \lambda)$-measure sequence which has a constructing embedding.

\item  $\mathcal{U}^{\rm sup}_{n+1} = \{ u \in \mathcal{U}_{n}: \mbox{for all nonzero $\alpha < \lh(u)$,
  $u(\alpha)$ concentrates on $\mathcal{U}_n$} \}$.

\item $\mathcal{U}^{\rm sup}_{\infty} = \bigcap_{n<\omega}\mathcal{U}_n$.

\end{itemize}

\end{definition}

As in Section \ref{vanillaradin},   it follows from the countable completeness of the measures in $u$ that if
 $u \in \mathcal{U}^{\rm sup}_{\infty}$
every measure in $u$ concentrates on $\mathcal{U}^{\rm sup}_{\infty}$.

\begin{definition}
Given a measure sequence $u\in \mathcal{U}^{\rm sup}_{\infty}$, we let
  $\kappa_u = \min(ON \setminus u(0))$ and $\lambda_u = \ot( u(0) )$.
\end{definition}

\begin{definition}
Given a $(\kappa, \lambda)$-measure sequence $u$, a non-zero  ordinal $\alpha < \lh(u)$ is  a \emph{weak repeat point} for $u$
 if and only if  for all $X \in u(\alpha)$ there exists a non-zero $\beta < \alpha$ such that $X \in u(\beta)$.
\end{definition}

Failure to be a weak repeat point is witnessed by a ``novel'' subset of $S(\kappa, \lambda)$,
 and $\vert S(\kappa, \lambda) \vert = \lambda^{<\kappa}$, so we have the following
 easy result.

\begin{lemma} \label{sccrepeatlemma}
A $(\kappa,\lambda)$-measure sequence of length $(2^{\lambda^{<\kappa}})^+$ contains a weak repeat point.
\end{lemma}

   For the purposes of our main result, we need a good $(\kappa, \kappa^+)$-measure sequence
   with a weak repeat point.

\begin{lemma} \label{gettingulemma}
Let $GCH$ hold and let $j: V \rightarrow M$ witness that $\kappa$ is $\kappa^{+3}$-supercompact.
Then there exists a $(\kappa, \kappa^+)$-measure
sequence $u$ such that $u \in \mathcal{U}^{\rm sup}_{\infty}$ and $u$ has a weak repeat
point.
\end{lemma}

\begin{proof}
  Evidently $M$ contains every $(\kappa, \kappa^+)$-measure sequence of length
   less than $\kappa^{+3}$, so that the construction of $u_j$ runs for
   at least $\kappa^{+3}$ steps. By Lemma \ref{sccrepeatlemma} a weak repeat point $\alpha$ appears
   before stage $\kappa^{+3}$, so it suffices to check that $u \in \mathcal{U}^{\rm sup}_{\infty}$, where
   $u = u_j \restriction (\alpha +1)$. Clearly $j$ is a constructing embedding for $u$.
   To verify that $u$ is good, we will first use a reflection argument to show
   that $u$ has a constructing embedding in $M$.

   Let $W$ be the supercompactness measure on $P_\kappa (\kappa^{+2})$ induced by
   the embedding $j$ and let $j_0 : V \rightarrow M_0 = Ult(V, W)$ be the usual
   ultrapower map. The following claims are all standard and easy to verify:
\begin{enumerate}
\item $W \in M$.

\item If we define a map $k$ from $M_0$ to $M$ by setting
   $k : [F]_W \mapsto j(F)(j[\kappa^{+2}])$ then $k$ is elementary,
   and $k \circ j_0 = j$.

\item $\kappa^{+3} = (\kappa^{+3})^{M_0} < \crit(k) = (\kappa^{+4})^{M_0} < \kappa^{+4} = (\kappa^{+4})^M$.

\item The sets $j_0[\kappa^+]$ and $j_0[\kappa^{+2}]$ are in $M_0$,
   $k(j_0[\kappa^+]) = j[\kappa^+]$, and $k(j_0[\kappa^{+2}]) = j[\kappa^{+2}]$.

\item Since $\kappa^{+3} \subseteq \rge(k)$, it follows easily from $GCH$ that
   $H_{\kappa^{+3}} \subseteq \rge(k)$, and so that  $k \restriction H_{\kappa^{+3}} = id$.

\item Let $j_0^*: M \rightarrow M^* = Ult(M, W)$ be the ultrapower map defined from $W$ in $M$. For every $a \in M$,
   ${}^{P_\kappa (\kappa^{+2})} TC({a}) \subseteq M$ and so $j_0(M) = j_0^*(M)$, that
   is $j_0^* = j_0 \restriction M$.
\end{enumerate}
Now let $u^* = u_{j_0^*}^M$, the measure sequence constructed by the embedding $j_0^*$ in the
   model $M$. We claim that the construction of $u^*$ proceeds for at least $\alpha +1$

   steps and that $u(\beta) = u^*(\beta)$ for all $\beta$ with $0 < \beta \le \alpha$. At the start,
   $u^*(0) = j_0^*[\kappa^+] = j_0[\kappa^+]$, and we now proceed by induction
   on $\beta$ with $0 < \beta \le \alpha$. Note that all the models $V$, $M$, $M_0$ and $M^*$ agree
   on the computation of $P(S(\kappa, \kappa^+))$.

   Suppose that $u^*(\eta) = u(\eta)$ when $0 < \eta < \beta$. Since $M_0$ is closed under
   $\kappa^{+2}$-sequences, $u^* \restriction \beta \in M_0$ and the properties of $k$
   listed above imply that $k(u^* \restriction \beta) = u \restriction \beta$.
   Now for every $X \subseteq  S(\kappa, \kappa^+)$ we have that

$\hspace{4.cm}$  $X \in u^*(\beta)   \iff   u^* \restriction \beta \in j_0^*(X)$
   
$\hspace{5.7cm}$     $\iff   u^* \restriction \beta \in j_0(X)$ 
   
$\hspace{5.7cm}$     $\iff   u \restriction \beta \in j(X) $
             
$\hspace{5.7cm}$     $\iff   X \in u(\beta),$

   where the equivalences follow respectively from the definition of $u^*$, the agreement between
   $j_0$ and $j_0^*$, the properties of $k$, and the definition of $u$.

    We can now verify that $u \in {\mathcal U}^{\rm sup}_1$. This holds because for each
    $\beta \le \alpha$ the sequence $u \restriction \beta$ has a constructing embedding in $M$,
    and so by the recursive definition of $u$ the measure $u(\beta)$ concentrates on the
    class of sequences with a constructing embedding.

   The rest of the argument is straightforward. We start by observing that since
   $\kappa^{+3} = (\kappa^{+3})^M$, every measure in $u$ concentrates on the
   set of measure sequences $x$ such that $x(0) \in P_\kappa (\kappa^+)$
    and  $\lh(x) < \kappa_x^{+3}$. By the agreement between $V$ and $M$, a routine
   induction shows that for all $n$ and for all such measure sequences
   $x$ we have $x \in \mathcal{U}_{n} \iff x \in \mathcal{U}_{n}^M$.

   We now establish by induction on $n \ge 1$ that $u \in \mathcal{U}^{\rm sup}_n$. We just did the base
   case $n=1$, so suppose that we established $u \in \mathcal{U}^{\rm sup}_n$. By definition,
   $u \in \mathcal{U}^{\rm sup}_{n+1}$ if and only if $u(\beta)$ concentrates on $\mathcal{U}^{\rm sup}_n$ for
   $0 < \beta \le \alpha$, that is if $u \restriction \beta \in (\mathcal{U}^{\rm sup}_n)^M$ for
   all such $\beta$.  By definition $u \restriction \beta \in (\mathcal{U}^{\rm sup}_n)^M$ if and
   only if $u(\gamma)$ concentrates on $(\mathcal{U}^{\rm sup}_{n-1})^M$ for all $\gamma$ with  $0 < \gamma < \beta$,
   and by the remarks in the previous paragraph this amounts to verifying that
   $u(\gamma)$ concentrates on $\mathcal{U}^{\rm sup}_{n-1}$ which is true since $u \in \mathcal{U}_{n}$.
\end{proof}

   Recall that in  Definition \ref{sccconstructingembedding}, we permitted
   that $w(0) \neq j[\lambda]$ in the definition of ``$j$ is a constructing embedding for $w$''.
   In  light of the proof of Lemma \ref{gettingulemma}, this may seem less surprising:
   the set $j[\lambda]$ is necessarily an element of $M$, but may not be of the form
   $i[\lambda]$ for a suitable supercompactness embedding {\em defined in $M$.}

   Before defining supercompact Radin forcing, we need to define a family of ``type changing maps''.
   These are needed because if $x \in S(\kappa, \lambda) \cap {\mathcal U}^{\rm sup}_\infty$, then
   the measures in $x$ concentrate on $S(\kappa_x, \lambda_x)$ rather than $S(\kappa, \lambda)$.

\begin{remark} The type changing maps are functions from ordinals to ordinals, whose role is to
   change the type of a measure sequence $x$ via pointwise application to $x(0)$.
   Accordingly we will systematically abuse notation in the following way: whenever $\nu$
   is one of the type changing maps $\pi_v$, $\rho_v$ or $\sigma_{v w}$
   from the forthccoming Definitions \ref{piandrho} and \ref{sigma},
   and $x$ is a measure sequence then
\[
   \nu(x) = \langle  \nu[x(0)] \rangle^\frown \langle x(\beta) : 0 < \beta < \lh(x) \rangle.
\]
\end{remark}

  Recall that we defined an ordering $\prec$ on $A(\kappa, \lambda)$ by stipulating that
  $x \prec y$ if and only if $x \subseteq y$ and $\lambda_x < \kappa_y$.

\begin{definition} \label{precforseq}
  Let $v, w \in S(\kappa, \lambda)$. Then $w \prec v$ if and only if $w(0) \prec v(0)$,
  $\lh(w) < \kappa_v$ and $w(\beta) \in V_{\kappa_v}$ for all $\beta$ such that
  $0 < \beta < \lh(w)$.
\end{definition}

   In the sequel there will be situations where several spaces of the general form
   $S(\kappa, \lambda)$ appear at once. We will only compare sequences $v$ and $w$
   when they lie in the same such space, and the values of $\kappa$ and $\lambda$ should
   always be clear from the context.

\begin{definition} \label{piandrho}
   Let $v \in S(\kappa, \lambda)$. Then we define $\pi_v : v(0) \simeq \lambda_v$ to be the
   unique order isomorphism between $v(0)$ and $\lambda_v$, and also define
   $\rho_v : \lambda_v \simeq v(0)$ by $\rho_v = \pi_v^{-1}$.
\end{definition}

\begin{definition} \label{sigma}
   Let $v, w \in S(\kappa, \lambda)$ with $w \prec v$.
  Then we define $\sigma_{w v} : \lambda_w \rightarrow \lambda_v$ by
  $\sigma_{w v} = \pi_v \circ \rho_w$.
\end{definition}

   Informally $\sigma_{w v}$ is a ``collapsed'' version of the inclusion map from
   $w(0)$ to $v(0)$.   The following Lemmas are straightforward.

\begin{lemma} Let $v \in S(\kappa, \lambda)$. Then $\rho_v \restriction \kappa_v$ is the identity,
   $\rho_v(\kappa_v) \ge \kappa$, and
   for every  $w \in S(\kappa_v, \lambda_v)$ we have $\rho_v(w) \in S(\kappa, \lambda)$ with
   $\rho_v(w) \prec v$.
\end{lemma}

\begin{lemma} Let $v, w \in S(\kappa, \lambda)$ with $w \prec v$.
   Then $\sigma_{w v} \restriction \kappa_w$ is the identity,
   $\pi_v(w) \in S(\kappa_v, \lambda_v)$,
   and for every $x \in S(\kappa_w, \lambda_w)$ we have
   $\sigma_{w v}(x) \in S(\kappa_v, \lambda_v)$.
\end{lemma}

\begin{definition}
A {\em good pair} is a pair $(u, A)$ where $u \in \mathcal{U}^{\rm sup}_{\infty}$,  and either $\lh(u) = 1$ and
  $A = \emptyset$ or $\lh(u) > 1$, $A \subseteq \mathcal{U}^{\rm sup}_{\infty} \cap  S(\kappa_{u}, \lambda_{u})$
   and  $A$ is  $u$-large.
\end{definition}

We  define, for each $u \in  \mathcal{U}^{\rm sup}_{\infty}$ with length greater than $1$,
    the corresponding {\em supercompact Radin forcing} ${\mathbb R}^{\rm sup}_{u}$.

\begin{definition} \label{sccradincdn}
 Let $u \in  \mathcal{U}^{\rm sup}_{\infty}$ with length greater than $1$.
A condition in  ${\mathbb R}^{\rm sup}_{u}$ is a finite sequence
\[
     p=\langle (u^0, A^0), \ldots, (u^i, A^i), \ldots, (u^n, A^n) \rangle
\] where:
\begin{enumerate}
\item $u^n = u$.
\item Each $(u^i, A^i)$ is a good pair.
\item $u^i \in  S(\kappa_{u^n}, \lambda_{u^n})$ for  all $i < n$.
\item $u^i \prec u^{i+1}$ for all $i < n-1$.
\end{enumerate}

   The sequence $\langle (u^0, A^0), \ldots,  (u^{n-1}, A^{n-1}) \rangle$ is the
   {\em stem} of the condition, and $A^n$ is the {\em upper part.}

\end{definition}

\begin{definition} \label{sccradinextn}
Let
\[
       p = \langle (u^0, A^0), \ldots, (u^i, A^i), \ldots, (u^m=u, A^m) \rangle
\]
   and
\[
       q = \langle (v^0, B^0), \ldots, (v^j, B^j), \ldots, (v^n=u, B^n) \rangle
\]
     be in ${\mathbb R}^{\rm sup}_{u}$. Then $q \le p$ ($q$ is an extension of $p$) iff:

\begin{enumerate}

\item \label{sccext1} There exist natural numbers $j_0 < \ldots < j_m = n$ such that $v^{j_k} = u^k$ and $B^{j_k} \subseteq A^k$.

\item If $j$ is such that  $0 \le j \le n$ and $j \notin \{ j_0, \ldots, j_m \}$,
    and  $i$ is  least such that $\kappa_{v^j} < \kappa_{u^i}$, then:

\begin{itemize}

\item \label{sccext2a} If $i = m$, then $v^j \in A^m$ and $\rho_{v_j}[B^j] \subseteq A^m$.

\item \label{sccext2b} If $i < m$, then $\pi_{u^i}(v^j) \in A^i$ and
   $\sigma_{v^j u^i}[B^j] \subseteq A^i$.

\end{itemize}

\end{enumerate}

We also define $q \le^* p$ ($q$ is a direct extension of $p$) iff $q \le p$ and $n = m$.

\end{definition}

\begin{remark}   Any two conditions with the same stem are compatible. Just as  for supercompact Prikry forcing,
   that there are only $\lambda$ possible stems.
 Therefore  ${\mathbb R}^{\rm sup}_u$ satisfies the $\lambda^+-c.c.$.
\end{remark}

   The following Lemma shows that conditions can be extended in many ways by adding in new good pairs immediately
    below the top entry.

\begin{lemma} \label{addability1}
 Let $u \in {\mathcal U}^{\rm sup}_\infty$ with  $\lh(u) > 1$, and
  let
\[
     p = \langle (u^0, A^0), \ldots, (u^i, A^i), \ldots, (u^m=u, A^m) \rangle
\]
   be a condition in
  ${\mathbb R}^{\rm sup}_{u}$. Let $X$ be the set of sequences $v$ with the following property: there exists  a
   set $B$ such that $(v, B)$ is a good pair, and   if we set
\[
     q = \langle (u^0, A^0), \ldots, (u^i, A^i), \ldots, (u^{m-1}, A^m), (v, B),  (u^m=u, A^m) \rangle
\]
  then $q$ is a condition extending $p$. Then $X$ is $u$-large.
\end{lemma}

\begin{proof} Let $j$ be a constructing embedding for $u$.  To show $X$ is $u$-large we must show
  that $u_j \restriction \beta \in j(X)$ for all $\beta$ with $0 < \beta < \lh(u)$.  It is easy to see that
  $v \in X$ if and only if it satisfies the conditions:
\begin{enumerate}
\item $v \in A^m$.
\item $u^{m-1} \prec v$ (in case $m > 0$).
\item $\{  w \in S(\kappa_v, \lambda_v) : \rho_v(w) \in A^m \}$ is $v$-large.
\end{enumerate}

   Since $A_m$ is $u$-large, $A_m \in u(\beta)$ and so $u_j \restriction \beta \in j(A^m)$.
   It is easy to check that for any $x \in S(\kappa_u, \lambda_u)$, $j(x) \prec u_j \restriction \beta$;
   the main points are that $j( x(0) ) = j[ x(0) ]$ and $u_j(0) = j[\lambda_u]$.
   Since $u_j(0) = j[\lambda_u]$, $\rho_{u_j \restriction \beta} = j \restriction \lambda_u$, and
   so easily $\rho_{u_j \restriction \beta}(w) = j(w)$ for all $w \in S(\kappa_u, \lambda_u)$;
   so $A^m = \{ w \in S(\kappa_u, \lambda_u) : \rho_{u_j \restriction \beta}(w) \in j(A^m) \}$, and
   since $A^m$ is clearly $u_j \restriction \beta$-large we have that
   $u_j \restriction \beta \in j(X)$.
\end{proof}

   Lemma \ref{addability1} states that $u$-many of the elements of $A^m$ can be interpolated into the
   condition $p$. We can iterate this argument $\omega$ times, by setting $B^0 = A^m$ and then defining
   a decreasing sequence $\langle B^i : i < \omega \rangle$ such that each $B^i$ is a $u$-large set
   and every element of $B^{i+1}$ can be interpolated between $(u^{m-1}, A^{m-1})$ and $(u, B^i)$.
   Now let $A' = \bigcap_{i < \omega} B^i$. Since the  the measures in $u$ are countably complete,
   $A'$ is $u$-large and we obtain the following corollary of Lemma \ref{addability1}.

\begin{corollary} \label{addability1point5}
 Let $u \in {\mathcal U}^{\rm sup}_\infty$ with  $\lh(u) > 1$, and
  let
\[
     p = \langle (u^0, A^0), \ldots, (u^i, A^i), \ldots, (u^m = u, A^m) \rangle
\]
   be a condition in
  ${\mathbb R}^{\rm sup}_{u}$. Then there is a $u$-large set $A' \subseteq A^m$ such that
   if we set
\[
      p' = \langle (u^0, A^0), \ldots, (u^i, A^i), \ldots, (u^m = u, A') \rangle,
\]
 then  for every $v \in A'$ there exists $q \le p'$ of the form
\[
     q = \langle (u^0, A^0), \ldots, (u^i, A^i), \ldots, (u^{m-1}, A^m), (v, B),  (u^m = u, A') \rangle.
\]
\end{corollary}

 Similarly, whenever $\lh(u^i) > 1$ there are many candidates for interpolation
   between $(u^{i-1}, A^{i-1})$ and $(u^i, A^i)$ in an extension of $p$.
 Before we prove this, we state a very easy but useful factoring lemma.

\begin{lemma} \label{sccfactor}
 Suppose that
\[
    p=\langle (u^0, A^0), \ldots, (u^i, A^i), \ldots, (u^m=u, A^m) \rangle \in {\mathbb R}^{\rm sup}_{u}
\] and  $i < m$ with $\lh(u^i) > 1$.  Let
\[
     p^{> i}= \langle  (u^{i+1}, A_*^{i+1}), \ldots, (u^m = u, A_*^m)  \rangle,
\]
   where $A^j_* = A^j$ for all $j > i+1$ while
   $A^{i+1}_* = \{ w \in A^{i+1} :  u^i \prec \rho_{u^{i+1}}(w) \}$.
   Let
\[
     p^{\le i}= \langle  (u^0_*, A^0), \ldots, (u^{i-1}_*, A^{i-1}), (u^i, A^i) \rangle,
\]
    where  $u^j_* = \pi_{u^i}(u^j)$ for each $j < i$.

 Then  $p^{\le i} \in {\mathbb R}^{\rm sup}_{u^i}$, $p^{> i} \in {\mathbb R}^{\rm sup}_{u}$ and there exists
\[
    i : {\mathbb R}^{\rm sup}_{u^i} \downarrow p^{\le i} \times {\mathbb R}^{\rm sup}_{u} \downarrow p^{> i} \rightarrow {\mathbb R}^{\rm sup}_u
\]
which is an isomorphism between its domain and a dense subset of ${\mathbb R}^{\rm sup}_u \downarrow p$.
\end{lemma}

\begin{proof}
 Let $q_0 \le p^{\le i}$ and $q_1 \le p^{>i}$,
 say $q_0 = \langle (v^0, B^0), \ldots, (v^j = u^i, B^j) \rangle$ and
 $q_1 = \langle (v^{j+1}, B^{j+1}), \ldots (v^n = u, B^n) \rangle$.
 Define $i(q_0, q_1)  = \langle  (v^j_*, B^j) : 0 \le j \le n \rangle$
  where $v^k_* = \rho_{u^i}(v^k)$ for all $k < j$, $v^k_* = v^k$ for all $k \ge j$.
\end{proof}

\begin{lemma} \label{addability2}
 Let $u \in {\mathcal U}^{\rm sup}_\infty$ with  $\lh(u) > 1$, and
  let
\[
     p = \langle (u^0, A^0), \ldots, (u^i, A^i), \ldots, (u^m=u, A^m) \rangle
\]
  be a condition in
  ${\mathbb R}^{\rm sup}_{u}$. Let $i  < m$ with $\lh(u^i) > 1$,
   and let $Y$ be the set of sequences $v$ with the following property: there exists  a
   set $B$ such that $(v, B)$ is a good pair, and if we set
\[
    q = \langle (u^0, A^0), \ldots,  (u^{i-1}, A^{i-1}), (v, B),  (u^i, A^i), \ldots, (u^{m-1}, A^{m-1}),  (u^m = u, A^m) \rangle
\]
  then $q$ is a condition extending $p$. Then $\pi_{u^i}[Y]$ is $u^i$-large.
\end{lemma}

\begin{proof}
 As we should expect from Lemma \ref{sccfactor}, we prove this by
   considering addability in ${\mathbb R}^{\rm sup}_{u^i}$. To be more precise,
   by the argument of Lemma \ref{addability1} there are $u^i$-many
   $\bar v \in A^i$ such that some
   $(\bar v, B)$ can be interpolated between
   $(\pi_{u^i}(u^{i-1}), A^{i-1})$ and $(u^i, A^i)$ in $p^{\le i}$.
   For each such $\bar v$ and $B$, if we set $v = \rho_{u^i}(\bar v)$ then
   $(v, B)$ can be interpolated between
   $(u^{i-1}, A^{i-1})$ and $(u^i, A^i)$ in $p$.
\end{proof}

\begin{definition} \label{sccradinappears}
 Given $p \in {\mathbb R}^{\rm sup}_{u}$,
 $p=\langle (u^0, A^0), \ldots, (u^i, A^i), \ldots, (u^n=u, A^n) \rangle$ and $w \in \mathcal{U}^{\rm sup}_{\infty}$, we say $w$
  {\em appears in $p$} if and only if $w=u^i$ for some $i<n$.
\end{definition}

 Very much as for the Radin forcing ${\mathbb R}_u$ discussed in Section \ref{vanillaradin},
  the generic object for the supercompact Radin forcing ${\mathbb R}^{\rm sup}_u$  can be viewed as a sequence
  $\langle v_\alpha : \alpha < \mu \rangle$ where $\langle v_\alpha(0) : \alpha < \mu \rangle$ is a continuous
  $\prec$-increasing sequence with union $\lambda$. As before, the condition $p$ above carries the information that each $v_\alpha$
  either appears among the $u^i$ or is the first entry in a pair which can legally be added to $p$.

   The main technical fact that we will need is the Prikry lemma for supercompact Radin forcing, which states
   (as usual) that every question about the generic extension can be decided by a direct extension.
   Before giving the proof of the Prikry lemma, we need a suitable version of normality for good measure sequences.

\begin{definition}
   Let $s$ be a stem for ${\mathbb R}^{\rm sup}_u$,
  say $s = \langle (u^0, A^0), \ldots (u^{m-1}, A^{m-1}) \rangle$,
 and let $w \in S(\kappa_u, \lambda_u)$. Then
 we define $s \prec w$ if and only if either $s$ is empty or
 $u^{m-1} \prec w$.
\end{definition}

\begin{definition}
  Let $L$ be a set of stems for ${\mathbb R}^{\rm sup}_u$, and let
  $\langle A_s : s \in L \rangle$ be an $L$-indexed family of subsets of
  $S(\kappa_u, \lambda_u)$. Then the {\em diagonal intersection} of the family
  $\Delta_s A_s$ is defined to be $\{ w : \forall s \prec w \; w \in A_s \}$.
\end{definition}

The following lemma is a form of normality for the measures on a good measure sequence.

\begin{lemma} \label{sccdiagintlemma}  Let $u \in {\mathcal U}^{\rm sup}_\infty$ with $\lh(u) > 1$ and let $0 < \beta < \lh(u)$.
    Let $L$ be a set of lower parts for ${\mathbb R}^{\rm sup}_u$ and let $\langle A_s: s \in L \rangle$ be a sequence such that
    $A_s \in u(\beta)$ for all $s$.
  Then  $\Delta_s B_s = \{ w  : \forall s \prec w \; w \in B_s \} \in u(\beta)$.
\end{lemma}

\begin{proof}
    As we remarked in the course of proving Lemma \ref{addability1}, if $j$ is constructing for $u$
   then $j(w) \prec u_j \restriction \beta$ for every $w \in S(\kappa_u, \lambda_u)$. It
   is easy to check that the converse also holds: if $v \in j(S(\kappa_u, \lambda_u))$
   and $v \prec u_j \restriction \beta$ then $v = j(w)$ for some $w \in S(\kappa_u, \lambda_u)$.
   It is now routine to verify that $u_j \restriction \beta \in j(\Delta_s B_s)$, so that
   $\Delta_s B_s \in u(\beta)$.
\end{proof}

\begin{corollary}
   The diagonal intersection of a family of $u$-large sets is $u$-large.
\end{corollary}

  The proof of the Prikry lemma follows the usual template for proving such lemmas. We include it
  to make this paper more self-contained and
  to confirm that our definition of {\em good measure sequence} causes no problems.
\begin{lemma} \label{sccprikrylemma}
  Let $\phi$ be a sentence of the forcing language and let $p \in {\mathbb R}^{\rm sup}_u$. Then there is
  $q \le^* p$ such that $q$ decides $\phi$.
\end{lemma}

\begin{proof} We start by reducing to the case when the stem of the condition $p$ is empty.
 To do this assume that we have the Prikry lemma for conditions with empty stems, let
 $p = \langle (u^0, A^0), \ldots, (u^i, A^i), \ldots, (u^n=u, A^n) \rangle$ with $n > 0$, and
 use Lemma \ref{sccfactor} with $m = n - 1$ to view the truth value of $\phi$ (which is an ${\mathbb R}^{\rm sup}_u \downarrow p$-name
 for an element of $2$)  as a ${\mathbb R}^{\rm sup}_{u} \downarrow p^{> n-1}$-name for a
  ${\mathbb R}^{\rm sup}_{u^{n-1}} \downarrow p^{\le n-1}$-name for an element
 of $2$. Since $\lambda_{u^{n-1}} < \kappa_u$ and $p^{> n-1}$ has an empty stem, we may shrink the measure one
 set in $p^{> n-1}$ to determine the value of
 this  ${\mathbb R}^{\rm sup}_{u^{n-1}}/p^{\le n-1}$-name; working downwards and repeating the argument a further $n$ times, we end with the
 conclusion of the Prikry lemma.

 So we let $p=\langle (u, A) \rangle$. Let $I$ be the set of stems $s$ such that there is $B \subseteq A$
 with $s^\frown \langle (u, B) \rangle$ deciding $\phi$, and for each $s \in I$ let $A_s$ be such a set $B$.
 Let $A_1 = \Delta_{s \in I} A_s$, then easily $A_1$ has the following property: for every stem $s$, if there
 exists $B \subseteq A$ such that $s^\frown \langle (u, B) \rangle$ decides $\phi$, then $s \frown (u, A_1)$ decides
 $\phi$.

 Now for each stem $s$, we partition $A_1$ into three parts: $A^s_{1,0}$ is the set of $w \in A_1$ such that for some
 $w$-large $B \subseteq V_{\kappa_w}$
\[
     s^\frown \langle (w, B) \rangle^\frown \langle (u, A_1) \rangle \Vdash \phi,
\]
  $A^s_{1, 1}$ is the set of $w \in A_1$ such that for some $w$-large $B$ the condition
\[
     s^\frown \langle (w, B) \rangle^\frown \langle (u, A_1) \rangle \Vdash \neg\phi,
\]
 and $A^s_{1, 2}$ is the remainder of $A_1$.
 For each $\alpha$ with $0 < \alpha < \lh(u)$,
  let $B^\alpha_s$ be $A^s_{1, n}$ for the unique $n$ such that $A^s_{1, n} \in u(\alpha)$,
 let $B^\alpha = \Delta_s B^\alpha_s$, and let $A_2 = \bigcup_\alpha B^\alpha$.
 By construction $A_2$ has the following property: if
\[
    s^\frown \langle (w, B) \rangle^\frown \langle (u, C) \rangle \le s^\frown \langle (u, A_2) \rangle
\]
   is a condition which  forces $\phi$, then there exists $\alpha$ such that
 \[
     \{ w' : \exists B' \;  \mbox{$s^\frown \langle (w', B') \rangle^\frown \langle (u, A_1) \rangle$ forces $\phi$} \}
\]
 is $u(\alpha)$-large, and similarly for $\neg \phi$.

 To finish the argument, we fix a condition $t^\frown \langle (u, C) \rangle \le \langle (u, A_2) \rangle$ which decides
 $\phi$ with $\lh(t)$ minimal, and argue that $t$ must be empty. If not
 let $t = s^\frown \langle (w, B) \rangle$, and assume (without loss of generality) that
 $s^\frown \langle (w,  B) \rangle^\frown \langle (u, C) \rangle$ forces $\phi$.
 By construction we find $\alpha$ and a function $f$ with $\dom(f) \in u(\alpha)$
 such that
 $s^\frown \langle (v,  f(v)) \rangle \frown \langle (u, A_2) \rangle$ forces $\phi$ for every $v \in \dom(f)$.

 We will now construct a set $A_3 \subseteq A_2$ such that every extension
 of $s^\frown \langle (u, A_3) \rangle$ is compatible with some condition of the form
 $s^\frown \langle (v, f(v)) \rangle^\frown \langle (u, A_2) \rangle$. This property implies that
$s^\frown \langle (u, A_3) \rangle$ forces $\phi$, contradicting the minimal choice of $\lh(t)$.
   We note that by the definition of extension in the forcing ${\mathbb R}^{\rm sup}_u$,
  we may assume from this point on that $s$ is empty.

We define various subsets of $A_2$:
\begin{itemize}

\item  For $v' \in A_2$,  $Y(v')$ is the set of $v \in \dom(f) \cap A_2$ such that
 $v' \prec v$, $\pi_v(v') \in f(v)$,  and $\{ x \in S(\kappa_{v'}, \lambda_{v'}) : \sigma_{v' v}(x) \in f(v) \}$
  is $v'$-large.

\item $X$ is the set of $v' \in A_2$ such that $Y(v')$ is $u(\alpha)$-large.

\item  $Y$ is the diagonal intersection  $\Delta_{v' \in X} Y(v')$.

\item $Z$ is the set of $w \in A_2$ such that
  $\{ x \in S(\kappa_w, \lambda_w) : \rho_w(x) \in Y \}$ is $w(\beta)$-large for some
  $\beta$ with $0 < \beta < \lh(w)$.

\end{itemize}

  Let $A_3 = X \cup Y \cup Z$. We will verify  that $A_3$ is $u$-large.

\medskip

\begin{claim}  $X$ is $u \restriction \alpha$-large.
\end{claim}

\begin{proof} Let $j: V \rightarrow M$ be constructing for $u$, let $B = j(f)(u_j \restriction \alpha)$.
 Applying Lemma \ref{addability2} in $M$ to the condition $\langle (u_j \restriction \alpha, B), (j(u), j(A_2)) \rangle$
 we obtain exactly the conclusion that there are $u \restriction \alpha$-many $v'$ such that $u_j \restriction \alpha \in j(Y(v'))$.
\end{proof}

\begin{claim}
 $Y$ is $u(\alpha)$-large.
\end{claim}

\begin{proof} Immediate by Lemma \ref{sccdiagintlemma}.
\end{proof}

\begin{claim} $Z$ is $u(\beta)$-large for all $\beta$ with $\alpha < \beta < \lh(w)$.
\end{claim}

\begin{proof}: We verify that $u_j \restriction \beta \in j(Z)$.
    As we saw in the proof of Lemma \ref{addability1}, $\rho_{u_j \restriction \beta} = j \restriction \lambda_u$,
 and $\rho_{u_j \restriction \beta}(x) = j(x)$ for $x \in S(\kappa_u, \lambda_u)$, so that
 $\{ x \in S(\kappa_u, \lambda_u) : \rho_{u_j \restriction \beta}(x) \in j(Y) \} = Y$. Since $\alpha < \beta$ and $Y$ is
 $u(\alpha)$-large, we are done.
\end{proof}

 Now consider an arbitrary extension
\[
     q = \langle (u^0, A^0), \ldots (u^{m-1}, A^{m-1}), (u^m=u, A^m) \rangle
\]
 of $\langle (u, A_3) \rangle$. There are various cases: the first case is the most important one,
 in that we handle the other cases by making a further extension to get into the first case.

\begin{enumerate}

\item There exists $j$ such that $u^j \in Y$ and $u^i \in X$ for all $i < j$. In this case we can readily verify that
 the condition $\langle (u^j, f(u^j)), (u, A_2) \rangle$ is compatible with $q$: the main point to check is that
 each pair $(u^i, A^i \cap f(u^j))$ for $i < j$ can legally be added below $(u^j, f(u^j))$, which
 is immediate from the definitions of $X$ and $Y$.

\item For the least $j$ such that $u^j \notin X$, $u^j \in Z$. In this case by Lemma \ref{addability2}
 we can interpolate some sequence in $Y$ between $u^{j-1}$ and $u^j$, and reduce to the first case.

\item $u^j \in X$ for all $j$. In this case we can interpolate some sequence in $Y$
   between $u^{m-1}$ and $u$, and again reduce to the first case.

\end{enumerate}

\end{proof}

  The next result collects some useful information about the extension by ${\mathbb R}^{\rm sup}_u$
  where $u$ is a good $(\kappa, \lambda)$-measure sequence.

\begin{lemma} \label{sccportmanteau}
 Let $u \in {\mathcal U}^{\rm sup}_\infty$ be a  $(\kappa, \lambda)$-measure sequence with $\lh(u) > 1$,
 and let $G$ be ${\mathbb R}^{\rm sup}_u$-generic over $V$.
The following hold in $V[G]$:

\begin{enumerate}

\item Let $\vec w =\langle w_\alpha : \alpha < \mu \rangle$ enumerate
   $\{w: w$ appears in some $p \in G \}$, so  that we have $w_\alpha \prec w_\beta$ for all
  $\alpha, \beta$ with $\alpha < \beta < \mu$.
 Then $\mu$ is a limit ordinal with $\mu \le \kappa_u$,  and  $V[G]=V[\vec w ]$.

\item  Let $\vec C=\langle  w_\alpha(0) : \alpha < \mu  \rangle$.  Then $\vec C$
is a $\prec$-increasing and continuous sequence in $P_{\kappa_u}(\lambda_u)$.
 Furthermore if $\lh(u) \geq \kappa_u$, then $\mu = \kappa_u$.

\item  For each $\alpha < \mu$,  $\kappa_{w_\alpha} < \lambda_{w_\alpha} < \kappa_{w_{\alpha + 1}}$.

\item $\lambda = \bigcup_{\alpha < \mu} w_\alpha(0)$, in particular if $\lambda > \kappa$ then $\lambda$ is collapsed.

\item If we let $D = \{ \kappa_{w_\alpha(0)} : \alpha < \mu \}$ then $D$ is a club subset of $\kappa$.

\item For every limit ordinal $\beta < \mu$, $\langle w^*_\alpha : \alpha < \beta  \rangle$ is ${\mathbb R}^{\rm sup}_{w_\beta}$-generic over $V$,
  where $w^*_\alpha=\pi_{w_\beta}(w_\alpha)$, and $\vec{w} \restriction [\beta, \mu)$ is
  ${\mathbb R}^{\rm sup}_u$-generic over $V[\langle w^*_\alpha: \alpha < \beta  \rangle]$.

\item For every  $\gamma < \kappa$ and  $A \subseteq \gamma$ with $A \in V[\vec{w}]$, $A \in V[\langle w^*_\alpha: \alpha < \beta  \rangle]$.
   where $\beta < \mu$ is the largest ordinal such that $\kappa_{w_\beta} \le \gamma$.

\item  If $\gamma$ is a cardinal in $V$ such that $\kappa_{w_\beta} < \gamma \le \lambda_{w_\beta}$ for some limit $\beta < \mu$,
    then $\gamma$ is collapsed and in fact $V[G] \models \vert \gamma \vert = \kappa_{w_\beta}$. Other
   cardinals are preserved.

\end{enumerate}

\end{lemma}

\begin{proof}

   We will prove each claim in turn.

\begin{enumerate}

\item   As we noted above, $G$ is the set of conditions such that each $w_\alpha$ appears either in $p$
     or in some extension of $p$. It follows that $V[G]=V[\vec w ]$. By Lemma \ref{addability1} and an easy density argument,
     $\mu$ must be a limit ordinal.

\item  It follows from the definition of the forcing that $\vec C$ is $\prec$-increasing, and from
     Lemmas \ref{addability1} and \ref{addability2}  that $\vec C$ is continuous at every limit ordinal. An easy induction
    argument shows that $\mu \ge \omega^{\lh(u) - 1}$, in particular if $\lh(u) \ge \kappa_u$ then
    necessarily $\mu = \kappa_u$.

\item   We have that $w_\alpha$ and $w_{\alpha+1}$ are good measure sequences with $w_\alpha \prec w_{\alpha+1}$, and the
    conclusion is immediate.

\item   It follows from Lemma \ref{addability1} that $\vec C$ is cofinal in $(A(\kappa_u, \lambda_u), \prec)$,
 so in particular its union is $\lambda_u$.

\item   This follows from Lemmas \ref{addability1} and \ref{addability2}.

\item  This is immediate from Lemma \ref{sccfactor}.

\item   This follows from Lemma \ref{sccfactor} and Lemma \ref{sccprikrylemma}, together with the remark
   that for each $\alpha$ the measures in $w_\alpha$ are $\kappa_{w_\alpha}$-complete.

\item Since $\langle w^*_\alpha : \alpha < \beta  \rangle$ is ${\mathbb R}^{\rm sup}_{w_\beta}$-generic over $V$, the cardinal $\gamma$ is collapsed
   as claimed. Preservation of other cardinals follows from the analysis of bounded subsets of $\kappa$ in $V[G]$, and the
    fact that ${\mathbb R}^{\rm sup}_{w_\beta}$ is $\lambda_{w_\beta}^+$-c.c.~for all limit $\beta$.

\end{enumerate}

\end{proof}

\begin{lemma} \label{sccpres}

Let  $v \in \mathcal{U}^{\rm sup}_{\infty}$ be a
 $(\kappa, \lambda)$-measure sequence constructed from a suitable embedding $j$. Assume that
   $v$   has a weak repeat point $\alpha$ and $j$ witnesses that $\kappa$ is $\mu$-supercompact
   for some $\mu > \lambda$. Let $u = v \restriction \alpha$ and let $G$ be ${\mathbb R}^{\rm sup}_{u}$-generic over $V$.
 Then in $V[G]$, $\kappa$ remains $\mu$-supercompact.
\end{lemma}

\begin{proof}
We prove the lemma in a sequence of claims.

\begin{claim} \label{ernie}

If $ A \in u(\beta)$ for all $\beta$ with $0 < \beta < \alpha$, then $u \in j(A)$.

\end{claim}

\begin{proof}

It suffices to show that $A$ is $v(\alpha)$-large. Suppose not.
 Then $A^c =S(\kappa, \lambda) \setminus A \in v(\alpha)$, so for some $\beta$ with $0 < \beta < \alpha$,
 $A^c \in v(\beta) = u(\beta)$ which is in contradiction with $A \in u(\beta)$.

\end{proof}

Note that any condition $p \in {\mathbb R}^{\rm sup}_{u}$ is of the form ${p_d}^\frown \langle (u, A)  \rangle$, for some unique $p_d$ and $A$.

 By  Claim \ref{ernie} $u \in j(A)$, so we can form the condition
\[
    p^* = {q_d}^\frown  {\langle (u, A) \rangle}^\frown \langle (j(u), j(A))  \rangle \in j({\mathbb R}^{\rm sup}_u),
\]
 where $q_d$ is obtained from $p_d$ by type changing to make the above condition well-defined,
 which is to say (arguing as in Lemma \ref{addability1}) that $q_d = j(p_d)$.

 The following can be proved easily.

\begin{claim}
             $p^* \le j(p)$  in $j({\mathbb R}^{\rm sup}_u)$, and if $p \le q$ then $p^* \le q^*$.
\end{claim}

Given any condition $p={p_d}^\frown \langle (u, A)  \rangle \in {\mathbb R}_{u}$ and any $j(u)$-large set $E$, set
\[
    p^* \restriction E= {q_d}^\frown  {\langle (u, A) \rangle}^\frown \langle (j(u), j(A) \cap E)  \rangle.
\]

We now define $U$ on $P_{\kappa} \mu$ as follows:
  $X \in U$ if and only if there exist $p \in G$ and $E$ which is $j(u)$-large
  such that $p^*\restriction E \Vdash j[\mu] \in j(\dot{X})$,
where $\dot{X}$ is a name for $X$.

\begin{claim}
The above definition of $U$ does not depend on the choice of the name $\dot{X}$.
\end{claim}

\begin{proof}
Suppose that  $i_G( \dot{X_1} ) = i_G(\dot{X_2})$ and that $p^* \restriction E \Vdash j[\mu] \in j(\dot{X_1})$.
   Strengthening $p$ if necessary we may assume that $p \Vdash \dot{X_1} = \dot{X_2}$,
   so that by elementarity $j(p) \Vdash j(\dot{X_1}) = j(\dot{X_2})$.
   Now $p^* \restriction E \le p^* \le j(p)$, so that $p^* \restriction E \Vdash j[\mu] \in j(\dot{X_2})$.
\end{proof}

We show that $U$ is a normal measure on $P_\kappa \mu$. $U$ is easily seen to be a filter.

\begin{claim}
       $U$ is an ultrafilter.
\end{claim}

\begin{proof}
Let $\dot X$ name a subset $X$ of $P_\kappa \mu$.  Appealing to Lemmas \ref{sccprikrylemma} and Lemma \ref{sccfactor},
    we may find $E$ such that the condition $\langle (u, S(\kappa, \lambda)), (j(u), E) \rangle$
    forces that the truth value of $j[\mu] \in j(\dot X)$ is equal to a truth value
    for ${\mathbb R}^{\rm sup}_u$. Hence we may find $p \in G$ such that $p^* \restriction E$
    decides $j[\mu] \in j(\dot X)$, from which it follows that either $X \in U$ or $X^c \in U$.
\end{proof}

\begin{claim}
       $U$ is fine.
\end{claim}

\begin{proof}
Suppose that $\alpha < \mu$, $X = \{x \in P_\kappa(\mu): \alpha \in x \}$ and that $p \in {\mathbb R}^{\rm sup}_u$.
 It is clear that $p^* \Vdash j[\mu] \in j(\dot X)$, so $p$ forces $\dot{X} \in \dot{U}$.
\end{proof}

\begin{claim} \label{bert}
$U$ is normal.
\end{claim}

\begin{proof}
Let $F: P_\kappa \mu \rightarrow \lambda$ be regressive, that is $F(x) \in x$ for all non-empty
 $x \in P_\kappa \mu$, and let $\dot F$ name $F$.
  Appealing to Lemmas \ref{sccprikrylemma} and Lemma \ref{sccfactor}, together with the facts that
  $\mu < j(\kappa)$ and the measures in $U$ are $j(\kappa)$-complete,
  we may find $E$ such that the condition $\langle (u, S(\kappa, \lambda)), (j(u), E) \rangle$
  forces that for every $\alpha \in \mu$ the truth value of $j(\dot F)(j[\mu]) = j(\alpha)$
  is equal to a truth value for ${\mathbb R}^{\rm sup}_u$.
  Hence we may find $p \in G$ and $\alpha < \mu$ such that
  $p^* \restriction E \Vdash j(F)(j[\mu]) = j(\alpha)$, from which we see that
  $\{ x : F(x) = \alpha \} \in U$.
\end{proof}

\begin{claim}
$U$ is $\kappa$-complete.
\end{claim}

\begin{proof}
 This follows by a similar argument to the one we gave for normality in Claim \ref{bert}.
\end{proof}

It follows that  $U$ is a normal measure on $P_\kappa \mu$, and the lemma follows.
\end{proof}

\section{Projected forcing}  \label{sccprojforcing}

   As we saw in the last section, if $u \in {\mathcal U}^{\rm sup}_\infty$ is a  $(\kappa, \lambda)$-measure sequence
  and $G$ is generic for ${\mathbb R}^{\rm sup}_u$, then $(\alpha^+)^V < (\alpha^+)^{V[G]}$
  for a closed unbounded set of cardinals $\alpha < \kappa$. Using a sequence with a repeat point,
  we may also arrange that $\kappa$ is a large cardinal in $V[G]$.  We wish to find a submodel
  $V[G^\phi]$ of $V[G]$ such that $V[G^\phi]$ is a cardinal-preserving extension of $V$,
  and $HOD^{V[G]} \subseteq V[G^\phi]$.

  It will be technically convenient, and sufficient for the intended application,
  to assume from this point on that $\lambda = \kappa^+$. The measures in a good
  $(\kappa, \kappa^+)$-measure sequence will concentrate on a certain subset
  of $S(\kappa, \kappa^+)$, namely $\{ w \in S(\kappa, \kappa^+) : \lambda_w = \kappa_w^+ \}$.
  It follows that by working below a suitable condition in ${\mathbb R}^{\rm sup}_u$,
  we may assume that $\lambda_v = \kappa_v^+$ for every $v$ appearing on the generic sequence.

  In the interests of notational simplicity, we prefer to make a slight modification in certain definitions.
  From this point on we let $A(\kappa, \kappa^+)$ be the set of $x \in P_\kappa \kappa^+$ such that
  (as before) $\kappa_x = x \cap \kappa \in \kappa$ and $\kappa_x$ is inaccessible,  and (modified)
  $\lambda_x = \ot(x) = \kappa_x^+$. We then modify the definitions of
  $S(\kappa, \kappa^+)$ (Definition \ref{skappalambda}), $(\kappa, \kappa^+)$ measure sequence (Definition \ref{kappalambdasequence})
  good $(\kappa, \kappa^+)$-measure sequence (Definition \ref{sccgoodmeasuresequence}), and ${\mathbb R}^{\rm sup}_u$ in case
  $u$ is a good $(\kappa, \kappa^+)$ measure sequence (Definition \ref{sccradincdn}) accordingly.

  We will obtain $V[G^\phi]$ by defining a {\em projected sequence $\phi(u)$}, a {\em projected forcing} ${\mathbb R}^{\rm proj}_{\phi(u)}$, and an
  order-preserving map $\phi$  from ${\mathbb R}^{\rm sup}_u$ to  ${\mathbb R}^{\rm proj}_{\phi(u)}$, then
  arguing that $\phi[G]$ generates a  ${\mathbb R}^{\rm proj}_{\phi(u)}$-generic filter $G^\phi$.
  We note that our projected forcing is rather  different from the parallel construction
  of Foreman and Woodin \cite{foreman-woodin}. The reason is that we need our projected forcing
  to be as close to the supercompact Radin forcing  as possible, so that the quotient forcing is  sufficiently homogeneous.

 Given any $u \in {\mathcal U}^{\rm sup}_\infty$ we first define $\phi(u)$, and then we will define the projected forcing
 ${\mathbb R}_{\phi(u)}^{\rm proj}$.
\begin{definition}

Suppose $(u, A)$ is a good pair.

\begin{itemize}

\item $\phi(u) = \kappa_u^\frown \langle u(\zeta) : 0 < \zeta < \lh(u) \rangle$.

\item $\phi(A) = \{ \phi(v): v \in A    \}$.

\end{itemize}

Also let $\mathcal{U}_\infty^{\rm proj} =\{ \phi(u): u \in {\mathcal U}^{\rm sup}_\infty  \}$,
  and for $w \in \mathcal{U}_\infty^{\rm proj}$ let $\kappa_w = w(0)$.
 For $u \in {\mathcal U}_\infty^{\rm sup}$ and $0 < \alpha < \lh(u)$,
  let  $\phi( u(\alpha) )$ be the Rudin-Keisler projection of $u(\alpha)$ via the map $\phi$;
 similarly if $w \in {\mathcal U}^{\rm proj}_\infty$ and $0 < \alpha < \lh(w)$,
  let  $\phi( w(\alpha) )$ be the Rudin-Keisler projection of $w(\alpha)$ via the map $\phi$;

\end{definition}

 Note that $\phi( u(\alpha) ) \neq \phi(u)(\alpha) = u(\alpha)$, because the former is a measure on $V_\kappa$ and the latter is
  a measure on $S(\kappa, \kappa^+)$. In fact $\phi( u(\alpha) ) = \phi( \phi(u)(\alpha) )$.

\begin{definition}  Let $w, w' \in \mathcal{U}_\infty^{\rm proj}$. Then $w \prec w'$ if and only if $w \in V_{\kappa_{w'}}$.
\end{definition}

    It is routine to check that if $u, v \in \mathcal{U}_\infty^{\rm sup}$ with $u \prec v$, then $\phi(u) \prec \phi(v)$.

\begin{definition}

A {\em  good pair for projected forcing} is a pair $(w, B)$ where $w \in {\mathcal U}_\infty^{\rm proj}$,
   $B \subseteq {\mathcal U}_\infty^{\rm proj} \cap  V_{\kappa_w}$ and  $B \in \phi(w(\alpha))$ for all $\alpha$ with  $0 < \alpha < \lh(w)$.

\end{definition}

\begin{remark}

If $(u, A)$ is a good pair, then  $(\phi(u), \phi(A))$ is a good pair for projected forcing.

\end{remark}

Given $w \in {\mathcal U}_\infty^{\rm proj}$, we define the projected forcing ${\mathbb R}_w^{\rm proj}$.

\begin{definition}

Let  $w \in {\mathcal U}_\infty^{\rm proj}$.
A condition in  ${\mathbb R}_w^{\rm proj}$ is a finite sequence
\[
    p=\langle (w^0, B^0), \ldots, (w^i, B^i), \ldots, (w^m, B^m) \rangle
\]
   where:

\begin{enumerate}

\item $w^m = w$.

\item Each $(w^i, B^i)$ is a good pair for projected forcing.

\item $w_i \prec w_{i+1}$, for all $i<m$.

\end{enumerate}

\end{definition}

We now define the extension relation.
\begin{definition} \label{sccprojextdef}

Let
\[
     p = \langle (w^0, B^0), \ldots, (w^i, B^i), \ldots, (w^m = w, B^m) \rangle
\]
 and
\[
     q = \langle (x^0, C^0), \ldots, (x^j, C^j), \ldots, (x^n = w, C^n) \rangle
\]
    be in ${\mathbb R}_w^{\rm proj}$. Then $q \le p$ ($q$ is an {\em extension} of $p$) if and only if:

\begin{enumerate}

\item \label{sccprojext1} There exists an increasing sequence of natural numbers
   $j_0 < \ldots < j_m = n$ such that $x^{j_k} = w^k$ and $C^{j_k} \subseteq B^k$.

\item \label{sccprojext2} If $j$ is such that  $0 \le j \le n$ and  $j \notin \{ j_0, \ldots, j_m \}$,
   and if $i$ is  least such that $\kappa_{x^j}< \kappa_{w^i}$, then $x^j \in B^i$ and $C^j \subseteq B^i$.
\end{enumerate}

We also define $q \le^* p$ ($q$ is a {\em direct extension} of $p$) iff $q \le p$ and $m = n$.

\end{definition}

   As motivation, we consider the special case when $w = \phi(u)$ and $\lh(u) = 3$. Forcing
   with ${\mathbb R}^{\rm proj}_w$ below a suitable condition we will obtain a generic
   object $\vec w = \langle w_\alpha : \alpha < \omega^2 \rangle$ where:
\begin{enumerate}
\item $w_\alpha = \langle \kappa_\alpha \rangle$ for $\alpha = 0$ or a successor ordinal,
   $w_\alpha = \langle \kappa_\alpha, U_\alpha \rangle$ for $\alpha$ a limit ordinal.
\item $\langle \kappa_\alpha : \alpha < \omega^2 \rangle$ is increasing and cofinal in $\kappa_u$.
\item $U_\alpha$ is a $\kappa_\alpha^+$-supercompactness measure on $P_{\kappa_\alpha} \kappa_\alpha^+$.
\end{enumerate}
   As we see below, the generic extension $V[\vec w]$ preserves cardinals and can be viewed as
  a submodel of an extension $V[\vec u]$, where $\vec u = \langle u_\alpha : \alpha < \omega^2 \rangle$
  and $w_\alpha = \phi(u_\alpha)$ for each $\alpha$. The key idea is that the model $V[\vec w]$
  ``remembers'' the definitions of the forcing posets ${\mathbb R}^{\rm sup}_{u_\alpha}$ for each limit $\alpha$,
   and retains enough information to singularise $\kappa_\alpha$ for each such $\alpha$,
   but ``forgot''  the collapsing information that was present in the entries $u_\beta(0)$.

   The theory of ${\mathbb R}_w^{\rm proj}$ is very similar to that of ${\mathbb R}^{\rm sup}_u$, but the
  statements and proofs are  simpler because there is no need for the ``type changing'' maps.
  For example the following result is the analog of Lemmas \ref{addability1} and \ref{addability2}
  for ${\mathbb R}_w^{\rm proj}$, and can be proved in exactly the same way.

\begin{lemma} \label{projaddability}
 Let $w \in {\mathcal U}^{\rm proj}_\infty$ with  $\lh(w) > 1$, and
  let
\[
     p = \langle (w^0, B^0), \ldots, (w^i, B^i), \ldots, (w^m = w, B^m) \rangle
\]
  be a condition in
  ${\mathbb R}^{\rm proj}_w$. Let $i  \le  m$ with $\lh(w^i) > 1$,
   and let $Y$ be the set of sequences $v$ with the following property: there exists  a
   set $B$ such that $(v, B)$ is a good pair, and if we set
\[
    q = \langle (w^0, B^0), \ldots,  (w^{i-1}, B^{i-1}), (v, B),  (w^i, B^i), \ldots, (w^{m-1}, B^{m-1}),  (w^m = w, B^m) \rangle
\]
  then $q$ is a condition extending $p$. Then $Y$ is $w^i$-large.
\end{lemma}

\begin{definition} If $p \in {\mathbb R}_w^{\rm proj}$ with
 $p = \langle (w^0, B^0), \ldots, (w^i, B^i), \ldots, (w^m = w, B^m) \rangle$,
  then {\em $x$ appears in $p$} if and only if $x = w^i$ for some $i < m$.
\end{definition}

  It is easy to see that  ${\mathbb R}_w^{\rm proj}$ satisfies the $\kappa_w^+$-c.c.

\begin{lemma} \label{projclub}
Let $G$ be ${\mathbb R}_w^{\rm proj}$-generic over $V$, and let
\[
    D = \{ \kappa_x : \mbox{$x$ appears in some $p \in G$} \}.
\]
  Then $D$ is a club of $\kappa_w$. Furthermore if $\lh(w) \geq \kappa_w$, then  $\ot(D) = \kappa_w$.
\end{lemma}

\begin{proof}  This is an easy consequence of Lemma \ref{projaddability}.
\end{proof}

As before we have a factorization property for ${\mathbb R}_w^{\rm proj}$.

\begin{lemma} \label{sccprojfactor}
 Suppose that
\[
    p=\langle (w^0, B^0), \ldots, (w^i, B^i), \ldots, (w^m = w, B^m) \rangle \in {\mathbb R}^{\rm proj}_w
\]
    and  $i < m$ with $\lh(w^i) > 1$.  Let
\[
     p^{> i}= \langle  (w^{i+1}, B_*^{i+1}), \ldots, (w^m=w, B_*^m)  \rangle,
\]
   where $B^j_* = B^j$ for all $j > i+1$ while
   $B^{i+1}_* = \{ w \in B^{i+1} :  w^i \prec w \}$.
   Let
\[
     p^{\le i}= \langle  (w^0, B^0), \ldots, (w^{i-1}, B^{i-1}), (w^i, B^i) \rangle.
\]

 Then  $p^{\le i} \in {\mathbb R}^{\rm proj}_{w^i}$, $p^{> i} \in {\mathbb R}^{\rm proj}_w$ and there exists
\[
    i : {\mathbb R}^{\rm proj}_{w^i} \downarrow p^{\le i} \times {\mathbb R}^{\rm proj}_w \downarrow p^{> i} \rightarrow {\mathbb R}^{\rm proj}_w
\]
which is an isomorphism between its domain and a dense subset of ${\mathbb R}^{\rm proj}_u$.
\end{lemma}

\begin{proof}
 Let $q_0 \le p^{\le i}$ and $q_1 \le p^{>i}$,
 say $q_0 = \langle (v^0, B^0), \ldots, (v^j = w^i, B^j) \rangle$ and
 $q_1 = \langle (v^{j+1}, B^{j+1}), \ldots (v^n = w, B^n) \rangle$.
 Define $i(q_0, q_1)  = \langle  (v^j, B^j) : 0 \le j \le n \rangle$ .
 \end{proof}

\begin{lemma} \label{sccprojdiagintlemma}  Let $w \in {\mathcal U}^{\rm proj}_\infty$ with $\lh(w) > 1$ and let $0 < \beta < \lh(w)$,
    let $L$ be a set of lower parts for ${\mathbb R}^{\rm proj}_w$ and let $\langle A_s: s \in L \rangle$ be a sequence such that
    $A_s \in \phi(w(\beta))$ for all $s$.
  Then  $\Delta_s B_s = \{ w  : \forall s \prec w \; w \in B_s \} \in \phi(w(\beta))$.
\end{lemma}

\begin{proof}  Exactly like the proof of Lemma \ref{sccdiagintlemma}.
\end{proof}

\begin{lemma} \label{sccprojprikrylemma}
 ${\mathbb R}_w^{\rm proj}$ satisfies the Prikry property, that is to say every sentence of the forcing language can be decided
 by a direct extension.
\end{lemma}

\begin{proof} The proof is like the proof of Lemma \ref{sccprikrylemma}, using Lemma \ref{sccprojdiagintlemma}
   and the fact that $w = \phi(u)$ for some good sequence $u$ with a constructing embedding.
\end{proof}

   We sketch an alternative proof for Lemma \ref{sccprojprikrylemma} following the proof of Lemma \ref{projthm} in the
  next section.

\begin{lemma} \label{sccprojportmanteau}
   Let $w \in {\mathcal U}^{\rm proj}_\infty$ with $\lh(w) > 1$, and
  let $G$ be ${\mathbb R}_w^{\rm proj}$-generic over $V$.

\begin{enumerate}

\item  Let $\vec{x} = \langle x_\alpha : \alpha < \mu  \rangle$ enumerate
 $\{ x : \mbox{$x$ appears in some $p \in G$} \}$
so that   $x_\alpha \prec x_\beta$ for all $\alpha, \beta$ with $\alpha < \beta < \mu$.  Then
  $\mu$ is a limit ordinal with $\mu \le \kappa_w$, and
 $V[G] = V[\vec x]$.

\item Let $D = \{ x_\alpha(0) : \alpha < \mu \}$. Then $D$ is a club subset of $\kappa_w$.

\item  For every limit ordinal $\beta < \mu$, $\vec{x} \restriction \beta$ is ${\mathbb R}^{\rm proj}_{x_\beta}$-generic over $V$, and
 $\vec{x} \restriction [\beta, \mu)$ is ${\mathbb R}^{\rm proj}_w$-generic over $V[\vec{x} \restriction \beta]$.

\item Suppose $\gamma < \kappa$, $A \subseteq \gamma$, $A \in V[\vec{x}]$.
 Let $\beta < \mu$ be the largest ordinal such that $\kappa_{x_\beta} \le \gamma$. Then $A \in V[\vec{x} \restriction \beta]$.

\item The models $V$ and $V[G]$ have the same cardinals.

\end{enumerate}

\end{lemma}

\begin{proof}
   The proof is like the proof of Lemma \ref{sccportmanteau}, using Lemma \ref{sccprojprikrylemma}.
\end{proof}

\begin{lemma} \label{newsccpres}

Let GCH hold and let  $j:V \rightarrow M$ witness $\kappa$ is $\mu$-supercompact for some $\mu \ge \kappa^{+3}$. Let
 $v\in {\mathcal U}^{\rm sup}_\infty$ be a $(\kappa, \kappa^+)$-measure sequence constructed from $j$
 which has a weak repeat point $\alpha$. Let $u = v \restriction \alpha$ and let $G^{\phi}$ be ${\mathbb R}_{\phi(u)}^{\rm proj}$-generic over $V$.
 Then in $V[G^\phi]$, $\kappa$ remains $\mu$-supercompact.

\end{lemma}

\begin{proof}
  The proof is similar to that of Lemma \ref{sccpres}, using Lemmas \ref{sccprojfactor} and \ref{sccprojprikrylemma}.
\end{proof}

\subsection{Weak projection} \label{sccweakproj}

Suppose that $u \in {\mathcal U}^{\rm sup}_\infty$ and consider the forcing notions ${\mathbb R}^{\rm sup}_{u}$ and
    ${\mathbb R}_{\phi(u)}^{\rm proj}$.
 We define a map $\phi: {\mathbb R}^{\rm sup}_u \rightarrow  {\mathbb R}_{\phi(u)}^{\rm proj}$
in the natural way,  by
\begin{multline}
 \phi: \langle (u^0, A^0), \ldots, (u^i, A^i), \ldots, (u^n, A^n) \rangle \mapsto \\
 \langle (\phi(u^0), \phi(A^0)), \ldots, (\phi(u^i), \phi(A^i)), \ldots, (\phi(u^n), \phi(A^n)) \rangle
\end{multline}

   The map $\phi$ is possibly not a projection in the classical sense. The problem (in a representative special case)
    is that when we extend a condition
   $\phi(p)$ by drawing a $\prec$-increasing sequence $w^0 \prec \ldots \prec w^n$ from
   $\phi(A)$ where $(u, A)$ is the last pair in $p$, so that each sequence $w^i$ is of form $\phi(v)$ for some $v \in A$,
   there may not exist a  $\prec$-increasing sequence $v^0 \prec \ldots \prec v^n$ from $A$ with
   $\phi(v^i) = w^i$.

   We will show that $\phi$ has a weaker property, introduced by Foreman and Woodin \cite{foreman-woodin},
   which is sufficient for our purposes.

\begin{definition} \label{weakprojdef} Let $\mathbb P$ and $\mathbb Q$ be forcing posets.
 $\psi:\mathbb{P} \rightarrow \mathbb{Q}$ is a {\em weak projection} if and only if $\psi$ is order preserving and
 for all $p \in \mathbb{P}$ there is $p^* \le p$ such that for all $q \le \psi(p^*)$ there exists $p' \le p$ such that $\psi(p') \le q$.
\end{definition}

   It is easy to see that if $\psi :{\mathbb P} \rightarrow {\mathbb Q}$ is a weak projection and $G$ is a $\mathbb P$-generic filter, then
   $\psi[G]$ generates a $\mathbb Q$-generic filter.

\begin{lemma} \label{projthm}

$\phi: {\mathbb R}^{\rm sup}_u \rightarrow {\mathbb R}^{\rm proj}_{\phi(u)}$ is a weak projection,
  and in fact satisfies a stronger property:
 for all $p \in {\mathbb R}^{\rm sup}_u$ there is $p^* \le^* p$ such that for all $q \le \phi(p^*)$
  there is $p' \le p$ such that $\phi(p') \le^* q$.

\end{lemma}

\begin{proof} It is easily seen that $\phi$ preserves
  both the ordering $\le$ and the direct extension ordering $\le^*$.

\begin{claim} \label{bigbird}

 Suppose that $(u, A)$ is a good pair. Let $A'$ be the set of
   $w \in A$ such that
 for all $x \in A$ with $\phi(x) \prec \phi(w)$
   there is $\bar{x} \in A$ such that $\phi(\bar{x})= \phi(x)$ and $\bar{x} \prec w$.
 Then $A'$ is $u$-large.

\end{claim}

\begin{proof} Suppose that  $j: V \rightarrow M$ constructs $u$ and let $0< \alpha < \lh(u)$.
 We need to show that $A' \in u(\alpha)$ or equivalently $u_j \restriction \alpha \in j(A')$.
Thus we need to prove the following:

\smallskip

\noindent $(*)$ If $x \in j(A)$ and $\phi(x) \prec \phi(u_j \restriction \alpha)$,
 then there is $\bar{x} \in j(A)$ such that $\phi(\bar{x})=\phi(x)$ and
  $\bar{x} \prec u_j \restriction \alpha$.

\smallskip

   Since $\phi(x) \prec \phi(u_j \restriction \alpha)$, $\phi(x) \in V_{\kappa_u}$, and it follows that
 we may choose  $\bar{x} \in \rge(j)$ such that $\bar{x} \in j(A)$ and $\phi(\bar x) = \phi(x)$.
 Since $\bar{x}\in \rge(j)$ we have $\bar{x} = j(z)$ for some $z \in A$,
 and then as in Lemma \ref{addability1} we have that $\bar{x} \prec u_j \restriction \alpha$.

\end{proof}

We are now ready to complete the proof of Lemma \ref{projthm} by showing that
 for all $p \in {\mathbb R}^{\rm sup}_u$ there is $p^* \le^* p$ such that for all $q \le \phi(p^*)$
  there is $p' \le p$ such that $\phi(p') \le^* q$.
 Using the factorization property from Lemma
   \ref{sccfactor}, it is sufficient to prove this  for the special case where $p$ has the form $\langle (u,A) \rangle$.

   Let us say that a sequence $w$ is {\em addable} to the good pair $(u, A)$ if
   there exists $B$  such that $\langle (w, B) , (u, A) \rangle$ is a condition
   extending $(u, A)$ in ${\mathbb R}^{\rm sup}_u$.
 By  Lemma  \ref{addability1point5} we may assume, without loss of generality, that every
 $v \in A$ is  addable to $(u, A)$.

 We will proceed by iterating the map $A \mapsto A'$ from  Claim \ref{bigbird}.
Let $A^{(\omega)}= \bigcap_{n < \omega} A^{(n)}$ where  $A^{(0)}=A$ and $A^{(n+1)}=(A^{(n)})'$ for all $n < \omega$.
We will show that $p^*= \langle (u,A^{(\omega)}) \rangle$ is as required. Let
\[
    q=\langle (w^0, B^0), \ldots, (w^i, B^i), \ldots, (w^n, B^n) \rangle \in {\mathbb R}^{\rm proj}_{\phi(u)},
\]
 with $q \le \phi(p^*)$. We will find a sequence of good pairs $(u^i, C^i)$ for $0 \le i \le n$, such that:
\begin{enumerate}
\item $u^n = u$.
\item $u^0  \prec u^1 \prec \ldots \prec u^{n-1}$.
\item $u^i \in S(\kappa_{u^n}, \lambda_{u^n})$.
\item $\phi(u^i) = w^i$ and $\phi(C^i) \subseteq B^i$ for $0 \le  i \le n$.
\item $u^i \in A^{(i)}$ for $i < n$.
\item $\langle  (u^i, C^i), (u, A) \rangle \le p = \langle (u, A) \rangle$
\end{enumerate}

 We choose $(u^{n-i}, C^{n-i})$ by induction on $i$ for $0 \le i \le n$.
 For $i = 0$ let $u^n = u$ and let $C^n \subseteq B$ be such that $\phi(C^n) \subseteq B^n$.
 For $i =1$ choose any $v \in A^{(\omega)}$ such that $\phi(v) = w^{n-1}$,
 and note that $v \in A^{(n-1)}$ and  $v$ is addable to $(u, A)$; set $u^{n-1} = v$ and choose
 $C^{n-1}$ such that
  $\langle  (u^{n-1}, C^{n-1}), (u, A) \rangle \le \langle (u, A) \rangle$
 and $\phi(C^{n-1}) \subseteq B^{n-1}$.

 Suppose now that $1 < i < n$ and we have chosen $(u^{n-i}, C^{n-i})$.
  Choose $v \in A^{(\omega)}$ such that $\phi(v) = w^{n-i-1}$.
  Since $u^{n-i} \in  A^{(n-i)} = (A^{(n-i-1)})'$  and $v \in A^{n-i-1}$,
   we can find $u^{n-i-1} \in A^{n-i-1}$ such that $\phi(u^{n-i-1}) = w^{n-i-1}$
  and $u^{n-i-1} \prec u^{n-i}$. Since $u^{n-i-1} \in A$ it is addable to
  $(u, A)$, hence we may find $C^{n-i-1}$ such that
  $\langle  (u^{n-i-1}, C^{n-i-1}), (u, A) \rangle \le \langle (u, A) \rangle$
 and $\phi(C^{n-i-1}) \subseteq B^{n-i-1}$.

Let $p' = \langle (u^0, C^0), \ldots, (u^i, C^i), \ldots, (u^n, C^n) \rangle$. Then $p' \in {\mathbb R}^{\rm sup}_u$,
  $p' \le p$ and $\phi(p') \le^* q$.

\end{proof}

Lemma \ref{projthm} allows us to give an alternative proof of Lemma \ref{sccprojprikrylemma}, the Prikry lemma
   for ${\mathbb R}^{\rm proj}_w$. We choose some $u$ such that $\phi(u) = w$ and argue that
   $\phi[ {\mathbb R}^{\rm sup}_u]$ is dense in ${\mathbb R}^{\rm proj}_{\phi(u)}$ under the direct extension
   relation $\le^*$. Since $\phi$ preserves the direct extension relation, it is now fairly straightforward to
   argue that the Prikry lemma for ${\mathbb R}^{\rm sup}_u$ implies the Prikry lemma for  ${\mathbb R}^{\rm proj}_{\phi(u)}$.

\begin{remark}   If $\langle u^i : i < \mu \rangle$ is the generic sequence for ${\mathbb R}^{\rm sup}_u$ added by $G$, then
  it is easy to see that $\langle \phi(u^i) : i < \mu \rangle$ is the generic sequence for ${\mathbb R}^{\rm proj}_{\phi(u)}$ added by $G^\phi$.
\end{remark}

\subsection{Homogeneity property}

\begin{lemma} \label{permutationfact}
  Let $u \in {\mathcal U}^{\rm sup}_\infty$.
  For all $p, q \in {\mathbb R}^{\rm sup}_{u}$, if $\phi(p)$ and $\phi(q)$ are compatible in the $\le^*$ ordering,
    then there exist $p^* \le^* p$ and  $q^* \le^* q$  such that
    ${\mathbb R}^{\rm sup}_u \downarrow p^* \simeq {\mathbb R}^{\rm sup}_u \downarrow q^*$.
\end{lemma}

\begin{proof}

We give the proof in a sequence of claims. We will ultimately induce an isomorphism between the cones
${\mathbb R}^{\rm sup}_u \downarrow p^*$ and ${\mathbb R}^{\rm sup}_u \downarrow q^*$ using a permutation
  of $\kappa_u^+$, and so we begin with a general discussion of such permutations.

 Let $\tau$ be a permutation of $\kappa^+$ with $\tau \restriction \kappa = id$ , then (in an abuse of notation) we
  define $\tau(y) = \tau[y]$ for $y \in P_\kappa \kappa^+$. Clearly
  $\tau(y) \in P_\kappa \kappa^+$. If $y \in A(\kappa, \kappa^+)$ then
  $y \cap \kappa = \tau(y) \cap \kappa$, but possibly $\ot(y) \neq \ot(\tau(y))$ and in this case
  $\tau(y) \notin A(\kappa, \kappa^+)$.
  However $\{ y \in A(\kappa, \kappa^+) : \tau(y) = y \}$ is large, in a sense to be made precise
  in Claim \ref{swedishchef} below.

  If $y \in S(\kappa, \kappa^+)$, then we let
 \[
     \tau(y) = {\langle \tau(y(0)) \rangle}^\frown \langle y(\alpha): 0 < \alpha < \lh(y)  \rangle.
\]

\begin{claim} \label{swedishchef}

Suppose that $v$ is a $(\kappa, \kappa^+)$-measure sequence  which has a generating embedding $j$.
 Then $\{ y \in S(\kappa, \kappa^+) : \tau(y) = y \}$ is $v$-large.

\end{claim}

\begin{proof}

   We need to show that $j(\tau)(u_j \restriction \alpha) = u_j \restriction \alpha$,
   which is immediate because $u_j(0) = j[\kappa^+]$ and this set is closed under
  $j(\tau)$.

\end{proof}

We now complete the proof of Lemma \ref{permutationfact}. Thus let $p, q \in {\mathbb R}^{\rm sup}_u$ be such that
  $\phi(p)$ and $\phi(q)$ are compatible in the $\le^*$ ordering.

 Then $p$ and $q$ have the same length., say $n$. Let
\[
     p = \langle  (u^0, A^0), \ldots, (u^i, A^i), \ldots, (u^n = u, A^n) \rangle
\]
 and
\[
     q = \langle  (v^0, B^0), \ldots, (v^i, B^i), \ldots, (v^n = u, B^n) \rangle.
\]

 By our assumption  $\phi(u^i) = \phi(v^i) = w^i$ say. Let  $\kappa_{u^i} = \kappa_{v^i} = \kappa_i$ say,
  and note that $\ot(u^i(0)) = \ot(v^i(0)) = \kappa_i^+$, $\lh(u^i) = \lh(v^i) = \lh(w^i)$,
    and $u^i(\alpha) = v^i(\alpha) = w^i(\alpha)$ for all $\alpha > 0$.
   Note also that $u^0(0) \subseteq u^1(0) \ldots  \subseteq u^{n-1}(0)$,
   and similarly $v^0(0) \subseteq v^1(0) \ldots  \subseteq v^{n-1}(0)$.

   We may now easily build a permutation $\tau$ of $\kappa_u^+$ such that $\tau \restriction \kappa_u = id$, and
   $\tau(u^i(0)) = v^i(0)$ for all $i$ with $0 \le i < n$. Note that $\tau(u^i) = v^i$ for each $i$.

Note that for each $i < n$, $\tau$ induces a permutation $\tau_i$ of $\kappa_i^+$ defined by
 $\tau_i = \pi_{v^i} \circ \tau \circ \rho_{u^i}$. By convention set $\tau_n = \tau$.
For each $i$ with $0 \le i \le n$, let
\[
    C^i = A^i \cap B^i \cap \{ y \in S(\kappa_i, \kappa_i^+) : \tau_i(y) = y \}.
\]
 By Claim \ref{swedishchef} and the remarks above, $C^i$ is both $u^i$-large and $v^i$-large.

Now let
\[
     p^* = \langle  (u^0, C^0), \ldots, (u^i, C^i), \ldots, (u^n = u, C^n) \rangle
\]
 and
\[
     q^* = \langle  (v^0, C^0), \ldots, (v^i, C^i), \ldots, (v^n = u, C^n) \rangle.
\]
   We will define a function $\alpha$ with domain ${\mathbb R}^{\rm sup}_u \downarrow p^*$
   as follows: if
\[
   r = \langle ({\bar u}^0, D^0), \ldots, ({\bar u}^j, D^j), \ldots, ({\bar u}^t = u, C^t) \rangle \le p^*,
\]
   then $\alpha(r)$ is the sequence obtained by replacing ${\bar u}^j$ by $\tau({\bar u}^j)$ for each $j$ with
   $0 \le j < t$.

   We will verify that $\alpha(r) \in {\mathbb R}^{\rm sup}_u \downarrow q^*$, and $\alpha$ is an isomorphism between
  ${\mathbb R}^{\rm sup}_u \downarrow p^*$ and ${\mathbb R}^{\rm sup}_u \downarrow q^*$.

\begin{claim} \label{bork} $\alpha(r)$ is a condition.
\end{claim}

\begin{proof}

 Clearly each
  $D^j$ is $\tau({\bar u}^j)$-large, because the measures of $\tau({\bar u}^j)$ are the same as those
  of ${\bar u}^j$. We need to check that $\tau({\bar u}^j) \in S(\kappa_{{\bar u}^j}, \kappa_{{\bar u}^j}^+)$
  and that the sequences $\tau({\bar u}^j)$ are increasing in the $\prec$-ordering.

 Note that each of the sequences $u^i$ appears as ${\bar u}^j$ for some $j$, and in this case
 $\tau({\bar u}^j) = \tau(u^i) = v^i$.

\begin{subclaim} \label{bork1}   If $u^{n-1} \prec y$ and $y \in C^n$, then
   $v^{n-1} \prec \tau(y) = y$. Moreover if $u^{n-1} \prec y_0 \prec y_1$ and $y_0, y_1 \in C^n$, then
   $\tau(y_0) \prec \tau(y_1)$.
\end{subclaim}

\begin{proof} As $y \in C^n$, $\tau(y) = y$ and so $v^{n-1}(0) = \tau(u^{n-1}(0)) \subseteq \tau(y)(0) = y(0)$,
  hence easily $v^{n-1} \prec y$. The second part is immediate since $\tau(y_i) = y_i$.
\end{proof}

\begin{subclaim} \label{bork2}  For all $i < n$, if $u^{i-1} \prec y \prec u^i$ and $\pi_{u^i}(y) \in C^i$, then
    $v^{i-1} \prec \tau(y) = \rho_{v^i} (\pi_{u^i}(y)) \prec v^i$ and in particular
    $\tau(y) \in S(\kappa_y, \kappa_y^+)$.
  Moreover if $u^{i-1} \prec y_0 \prec y_1 \prec u^i$ and $\pi_{u^i}(y_0), \pi_{u^i}(y_0) \in C^i$,
  then $\tau(y_0) \prec \tau(y_1)$.
\end{subclaim}

\begin{proof}  Since $\pi_{u_i}(y) \in C^i$ it is fixed by $\tau_i$, hence $\pi_{v_i}( \tau(y) ) = \pi_{u_i}(y)$.
  Also $u^{i-1} \prec y \prec u^i$ and so $v^{i-1}(0) = \tau(u^{i-1}(0)) \subseteq \tau(y)(0) \subseteq v^i(0) = \tau(u^i(0))$.
  Since the type-changing maps are order preserving, $\ot(\tau(y)) = \ot(y) = \kappa_y^+$ and
  hence  $\tau(y) \in S(\kappa_y, \kappa_y^+)$. The final part is routine.
\end{proof}

   Combining the results of subclaims \ref{bork1} and \ref{bork2}, we see that $\alpha(r)$ is a condition.
  In the course of proving the subclaims we obtained a description of the   of the map $\alpha$
    in terms of the type-changing maps,  which we will use freely below.
\end{proof}

\begin{claim} $\alpha(r) \le q^*$.
\end{claim}

\begin{proof}
  As we already mentioned each of the sequences $v^i$ appears as $\tau({\bar u}^j)$ for some
  $j$.   By Definition \ref{sccradinextn},  if ${\bar u}^j$ is not among the sequences of the form
  $u^i$ then  either ${\bar u}^j \in C^n$,
 or  $\pi_{u_i}({\bar u}^j) \in C^i$ where $u^{i-1} \prec {\bar u}^j \prec u^i$.
 In the former case we have  that $\tau({\bar u}^j) = {\bar u}^j \in C^n$,
 and in the latter case
 we have $\pi_{v_i}(\tau({\bar u}^j)) = \pi_{u_i}({\bar u}^j) \in C^i$.

 To finish the proof that $\alpha(r) \le q^*$, we must check that the sets
 $D^j$ behave correctly with respect to the sets $C^i$ and the type-changing maps.
 If ${\bar u}^j = u^i$ then $D^j \subseteq C^i$ because $r \le p^*$. Otherwise
 we distinguish as before the cases $u^{n-1} \prec {\bar u}^j \in C^n$
 and $\pi_{u_i}({\bar u}^j) \in C^i$ where $u^{i-1} \prec {\bar u}^j \prec u^i$.

 In the former case $\rho_{ {\bar u}^j }[D^j] \subseteq C^n$ by clause \ref{sccext2a} of
 Definition \ref{sccradinextn} for $r \le p^*$.
 The latter case is slightly more complicated because of the type changing.
 By clause \ref{sccext2b} of  Definition \ref{sccradinextn} for $r \le p^*$,
  $\sigma_{ {\bar u}^j u^i }[D^j] \subseteq C^i$. We have that
 $\tau( {\bar u}^j ) = \rho_{v^i} (\pi_{u^i}( {\bar u}^j  ))$, and
 since $v^i =  \rho_{v^i} (\pi_{u^i}( u^i ) ) $ and the type changing maps are order-preserving
 we see that $\sigma_{ {\bar u}^j u^i } = \sigma_{ \tau( {\bar u}^j ) v^i }$. It follows that
 $\sigma_{ \tau( {\bar u}^j ) v^i }[D^j] \subseteq C^i$ as required.

\end{proof}

 It is clear that $\alpha$ is bijective, with an inverse defined in the same way
 using the permutation $\tau^{-1}$. To finish the proof we must check that
 $\alpha$ is order-preserving.  Let $r \le p^*$ be a condition as above, and let
 $s = \langle ({\hat u}^0, E^0), \ldots, ({\hat u}^k, E^k), \ldots, ({\hat u}^{\hat t} = u, C^{\hat t}) \rangle \le r$,

\begin{claim} $\alpha(s) \le \alpha(r)$.
\end{claim}

\begin{proof}
 If  ${\hat u}^k$ appears among the sequences of form ${\bar u}^j$ or if ${\bar u}^t \prec {\hat u}^k$,
 then there are no new technical points in checking Definition \ref{sccradinextn}  at the pair
  $( \tau({\hat u}^k), E^k)$. So we assume that neither of these cases holds, let $j$ be least
 such that ${\hat u}^k \prec {\bar u}^j$, and observe that there is no new technical issue
  if ${\bar u}^j$ is among the sequences of form $u^i$. This leaves us with the cases
  where $u^{n-1} \prec {\hat u}^k \prec {\bar u}^n$, and
  $u^{i-1} \prec {\hat u}^k \prec {\bar u}^j \prec u^i$ for some $i < n$.

  In the first case $\tau({\hat u}^k) = {\hat u}^k$ and $\tau({\bar u}^j) = {\bar u}^j$,
  so it is easy to verify clause \ref{sccext2a} at $\tau({\hat u}^k)$ in Definition \ref{sccradinextn},
  using the same clause at ${\hat u}^k$ from  Definition \ref{sccradinextn} for  $s \le r$.

  In the second case we have that
  $\pi_{ {\bar u}^j }( {\hat u}^k ) \in D^j$ and $\sigma_{ {\hat u}^k {\bar u}^j }[E^k] \subseteq D^j$.
   We also have
  that  $\tau({\hat u}^k) = \rho_{v_i}(\pi_{u_i}({\hat u}^k))$ and
  $\tau({\bar u}^j) = \rho_{v_i}(\pi_{u_i}({\bar u}^j))$,
  so that $\pi_{ \tau({\bar u}^j) }( \tau({\hat u}^k) ) = \pi_{ {\bar u}^j }( {\hat u}^k )$
  and   $\sigma_{{\hat u}^k  {\bar u}^j} = \sigma_{\tau({\hat u}^k) \tau({\bar u}^j)}$.
  Clause \ref{sccext2b} at $\tau({\hat u}^k)$ in Definition \ref{sccradinextn}
  now follows from  the same clause at ${\hat u}^k$ from  Definition \ref{sccradinextn} for  $s \le r$.
\end{proof}

This completes the proof of Lemma \ref{permutationfact}.
\end{proof}

\begin{corollary}(Weak homogeneity).
Suppose $p, q \in {\mathbb R}^{\rm sup}_{u}$ and $\phi(p)$ is compatible with $\phi(q)$ in the $\le^*$-ordering.
   If $p \Vdash \psi(\alpha, \vec{\gamma})$, where
  $\alpha, \vec{\gamma}$ are ordinals, then it is not the case that $q \Vdash \neg \psi(\alpha, \vec{\gamma})$.
\end{corollary}

It follows that:

\begin{corollary} \label{maincorollary}
Suppose that $G$ is ${\mathbb R}^{\rm sup}_{u}$-generic and let $G^{\phi}$ be the filter generated by $\phi[G]$.  Then:

\begin{enumerate}

\item \label{otherbit} $G^{\phi}$ is ${\mathbb R}^{\rm proj}_{\phi(u)}$-generic over $V$.

\item \label{mainbit} $HOD^{V[G]} \subseteq V[G^{\phi}]$.

\end{enumerate}

\end{corollary}

\begin{proof}

Part \ref{otherbit} is immediate because $\phi$ is a weak projection.

For part \ref{mainbit} suppose that $a$ is a set of ordinals in $V[G]$ which is definable  in
 $V[G]$ with ordinal parameters. We show that $a$ belongs to $V[G^{\phi}]$.
    Write $a=\{ \alpha : V[G] \models \psi(\alpha, \vec{\gamma}) \}$. Then
\[
    a  = \{ \alpha: \mbox{$p \Vdash \psi(\alpha, \vec{\gamma})$ for some $p\in G$} \}.
\]

   Let $\langle w_i: i < \mu \rangle$ be the generic sequence  induced by $G$, so that
   $\langle \phi(w_i) : i < \mu \rangle$ is the generic sequence induced by $G^\phi$. Let
   $H$ be the set of conditions $q$ in ${\mathbb R}^{\rm sup}_u$ such that for every finite set
   $b \subseteq \mu$ with $\{ i < \mu : \mbox{$\phi(w_i)$ appears in $\phi(q)$} \} \subseteq b$, there
   is $r \le q$ such that $\{ i < \mu : \mbox{$\phi(w_i)$ appears in $\phi(r)$} \} =  b$.
   Clearly $G \subseteq H$ and $H \in V[G^{\phi}]$.

   We claim that
\[
     a = \{\alpha:  \mbox{$p \Vdash \phi(\alpha, \vec{\gamma})$ for some $p \in H$ }   \}.
\]
   Clearly if $\alpha \in a$ then it is a member of the set on the right hand side, so suppose for a contradiction that
   there exist $p$ and $q$ such that $p \in G$, $q \in H$, $q \Vdash \phi(\alpha, \vec{\gamma})$
   and $p \Vdash \neg \phi(\alpha, \vec{\gamma})$. Let $b = \{ i < \mu : \mbox{$w_i$ appears in $p$ or $\phi(w_i)$ appears in $q$} \}$.
   Since $q \in H$ we may  find $q' \le q$ such that $b = \{ i < \mu : \mbox{$\phi(w_i)$ appears in $\phi(q')$} \}$,
   and since $p \in G$ we may find $p' \le p$ such that $b = \{ i < \mu : \mbox{$w_i$ appears in $p'$} \}$.

   It follows that $\phi(p')$ and $\phi(q')$ are compatible in the $\le^*$ ordering,
   contradicting Corollary \ref{maincorollary}.
\end{proof}

\section{The main theorem}

\begin{theorem} \label{mainthm}
 Let  GCH holds and let $\kappa$ be  $\kappa^{+3}$-supercompact. Then there exists a generic extension
  $W$ of $V$ in which $\kappa$ remains strongly inaccessible and $(\alpha^+)^{HOD} < \alpha^+$ for every infinite
  cardinal $\alpha < \kappa$. In particular the rank-initial segment $W_\kappa$ is a model of ZFC in which
  $(\alpha^+)^{HOD} < \alpha^+$ for {\em every} infinite cardinal $\alpha$.
\end{theorem}

\begin{proof}
Let $j: V \rightarrow M$ witness that $\kappa$ is a $\kappa^{+3}$-supercompact cardinal.
 Let $v\in {\mathcal U}^{\rm sup}_\infty$ be a $(\kappa, \kappa^+)$-measure sequence constructed from $j$ which has a weak repeat point
 $\alpha$ and let $u = v \restriction \alpha$.

Consider the forcing notions ${\mathbb R}^{\rm sup}_u$ and ${\mathbb R}^{\rm proj}_{\phi(u)}$.
 Let $G$ be ${\mathbb R}^{\rm sup}_{u}$-generic over $V$ and let $G^{\phi}$ be  the induced generic filter for
 ${\mathbb R}^{\rm proj}_{\phi(u)}$.  Let $\mu = \kappa^{+3}$. From previous results:

\begin{itemize}

\item $\kappa$ remains $\mu$-supercompact in $V[G]$ and in $V[G^{\phi}]$ (by Lemmas \ref{sccpres} and \ref{newsccpres}).

\item There exists $D \in V[G^{\phi}]$ a club subset of $\kappa$ such that for every limit
  point $\alpha$ of $D$,
   we have $(\alpha^+)^{V[G^{\phi}]} = (\alpha^+)^V < (\alpha^+)^{V[G]}$ (by Lemmas \ref{sccportmanteau} and \ref{sccprojportmanteau}).

\end{itemize}

 By part \ref{mainbit} of Corollary \ref{maincorollary} we have $HOD^{V[G]} \subseteq V[G^{\phi}]$, in particular for every limit point
 $\alpha$ of $D$ we have
\[
  (\alpha^+)^{HOD^{V[G]}} \le (\alpha^+)^{V[G^{\phi}]} =(\alpha^+)^{V} < (\alpha^+)^{V[G]}.
\]

Let $\langle \kappa_i: i< \kappa   \rangle$ be an increasing enumeration of $D$. Working in $V[G]$,
  let ${\mathbb Q}$ be the reverse Easton iteration for collapsing $\kappa_{i+1}$ to have cardinality $\kappa_{i}^+$
  for each $i<\kappa$, and let $H$ be $\mathbb{Q}$-generic over $V[G]$. By standard arguments about iterated forcing:

\begin{enumerate}

\item  $CARD^{V[G*H]} \cap (\kappa_0, \kappa) = \{ \kappa_{i}^+ : i<\kappa\} \cup \{ \kappa_i :\mbox{$i < \kappa$, $i$ is a limit ordinal}  \}$.

\item $\kappa$ remains inaccessible in $V[G*H]$.

\end{enumerate}

It also follows from  results of Dobrinen and Friedman \cite{dobrinen-friedman} that  $\mathbb{Q}$ is
 {\em cone homogeneous}, that is for all $p, q \in \mathbb{Q}$ there are $p^* \le p$, $q^* \le q$ and an isomorphism
 $\alpha: \mathbb{Q}/p^*  \rightarrow \mathbb{Q}/q^*$. Hence by \cite{dobrinen-friedman} we have
\[
    HOD^{V[G*H]} \subseteq HOD^{V[G]}.
\]
Finally let ${\mathbb P}=Col(\aleph_0, \kappa_0)^{V[G*H]}$ over $V[G*H]$, and let $K$ be  ${\mathbb P}$-generic over $V[G*H]$.
 It is now easily seen that $\kappa$ remains inaccessible in $V[G*H*K]$, and by homogeneity of $\mathbb{P}$
\[
   HOD^{V[G*H*K]} \subseteq HOD^{V[G*H]}.
\]
Hence
\[
    HOD^{V[G*H*K]} \subseteq HOD^{V[G*H]} \subseteq HOD^{V[G]} \subseteq V[G^{\phi}].
\]
Thus for all infinite cardinals $\alpha <\kappa$ of $V[G*H*K]$ we have
\[
    (\alpha^+)^{HOD^{V[G*H*K]}} \le (\alpha^+)^{V[G^{\phi}]} = (\alpha^+)^{V} < (\alpha^+)^{V[G*H*K]}.
\]
Let $W = V[G*H*K]$. Then $W$ is the required model and the theorem follows.
\end{proof}

\begin{remark}
   If we start with a cardinal $\kappa$ which is supercompact, then we may find $u$ such that
  $\kappa$ remains supercompact in the generic extension by ${\mathbb R}^{\rm sup}_u$.
  This gives a model where $\kappa$ is supercompact and $(\alpha^+)^{HOD} < \alpha^+$ for
  club-many $\alpha < \kappa$. We note that Dobrinen and Friedman \cite{dobrinen-friedman}
  gave a model where $\kappa$ is measurable with some normal measure $U$, and $(\alpha^+)^{HOD} < \alpha^+$ for
  $U$-many $\alpha < \kappa$.

  To preserve supercompactness we argue as follows.  We choose for each $\mu \ge \kappa^{+3}$ an embedding witnessing
  that $\kappa$ is $\mu$-supercompact, appeal to Lemmas \ref{gettingulemma} and \ref{sccpres}
  to find a $(\kappa, \kappa^+)$-measure sequence $u$ (depending on $\mu$) such that $\lh(u) < \kappa^{+++}$
  and ${\mathbb R}^{\rm sup}_u$ preserves the $\mu$-supercompactness of $\kappa$, and finally use the
  Axiom of Replacement to find a sequence $u$ such that  ${\mathbb R}^{\rm sup}_u$ preserves the $\mu$-supercompactness of $\kappa$
  for unboundedly many values of $\mu$.

\end{remark}

\begin{remark}

  We can show that $\kappa$ is measurable  in $W$ in Theorem \ref{mainthm}.
  Since $K$ is generic for small forcing, it suffices to show that
  $\kappa$ is measurable in $V[G * H]$.

  To do this let $W \in V[G]$ be a normal measure on $\kappa$, and let
  $i : V[G] \rightarrow N$ be the associated ultrapower map.
  Clearly $\kappa$ is a limit point of the club set $i(\mathbb D)$.
  By standard facts about Easton iterations the poset $\mathbb Q$ is
  $\kappa$-c.c.~and we may write $i(\mathbb Q) = {\mathbb Q} * {\dot{\mathbb R}}$ where
  $\mathbb R$ is the tail of the iteration. In $N[H]$ the first step of the  iteration $\mathbb R$
  is a $\kappa^+$-closed Levy collapse, and by standard arguments $\mathbb R$ is $\kappa^+$-closed.

  Since GCH holds in $V[G]$, we have that $2^\kappa = \kappa^+ = \vert i(\kappa) \vert$ in $V[G]$.
  In $N[H]$ the iteration $\mathbb R$ is a $i(\kappa)$-c.c.~poset of cardinality
  $i(\kappa)$ and so
\[
   N[H] \models \mbox{ $\mathbb R$ has $i(\kappa)$ maximal antichains.}
\]
  Since $V[G] \models {}^\kappa N \subseteq N$ and $H$ is generic for $\kappa$-c.c.~forcing,
  $V[G * H] \models {}^\kappa N[H] \subseteq N[H]$. It follows that working in $V[G * H]$ we may
  enumerate the antichains of $\mathbb R$ which lie in $N[H]$ in order type $\kappa^+$, and
  build a decreasing $\kappa^+$-chain of conditions in $\mathbb R$ that meets each of these antichains.
  This allows us to construct a filter $H^* \in V[G * H]$ which is $\mathbb R$-generic over $N[H]$.

  Since $\mathbb Q$ is an iteration with supports bounded in $\kappa$, $i`` H \subseteq H * H^*$.
  It follows that we may lift $i$ to obtain $i: V[G * H] \rightarrow N[H][H^*]$ definable
  in $V[G * H]$, and hence that $\kappa$ is measurable in $V[G * H]$.

\end{remark}

We conclude with some open questions:

\noindent

\begin{itemize}

\item   What is the exact consistency strength of the assertion
``$(\alpha^+)^{HOD} < \alpha^+$ for  every infinite cardinal $\alpha$''?

\item What is $HOD^{V[G]}$?

\item Is it consistent that all uncountable regular cardinals are inaccessible cardinals of $HOD$?

\item  Is it consistent that $\kappa$ is supercompact
  and $(\alpha^+)^{HOD} < \alpha^+$ for all cardinals $\alpha<\kappa$?

\end{itemize}

As a consequence of his ``HOD Conjecture'' (see \cite{woodin}), Woodin
has conjectured a negative answer to the last of these questions.

Department of Mathematical Sciences,
Carnegie Mellon University,
Pittsburgh PA 15213-3890,
USA.

E-mail address: jcumming@andrew.cmu.edu

Kurt G\"{o}del Research Center for Mathematical Logic,
W\"{a}hringer Strasse 25, 1090 Vienna, Austria.

E-mail address: sdf@logic.univie.ac.at

School of Mathematics, Institute for Research in Fundamental Sciences (IPM), P.O. Box:
19395-5746, Tehran-Iran.

E-mail address: golshani.m@gmail.com

\end{document}